\documentclass[a4paper,12pt,twoside,reqno]{amsart}

\usepackage[OT1]{fontenc}    %
\usepackage{type1cm}         
\usepackage[english]{babel}  %
\usepackage{amsthm}

\numberwithin{equation}{section}
\numberwithin{figure}{section}

\theoremstyle{plain}
\newtheorem{theorem}{Theorem}[section]

\newtheorem{proposition}[theorem]{Proposition}
\newtheorem{lemma}[theorem]{Lemma}

\theoremstyle{remark}
\newtheorem{remark}[theorem]{Remark}
\newtheorem{example}[theorem]{Example}
\newtheorem*{ack}{Acknowledgement}

\theoremstyle{definition}
\newtheorem{definition}[theorem]{Definition}

\newcommand{\R}{\mathbb{R}}

\newcommand{\N}{\mathbb{N}}

\newcommand{\EE}{\mathcal{E}}
\newcommand{\FF}{\mathcal{F}}
\newcommand{\GG}{\mathcal{G}}
\newcommand{\HH}{\mathcal{H}}

\newcommand{\MM}{\mathcal{M}}

\newcommand{\PP}{\mathcal{P}}
\newcommand{\QQ}{\mathcal{Q}}
\newcommand{\RR}{\mathcal{R}}

\newcommand{\hhh}{\mathtt{h}}
\newcommand{\iii}{\mathtt{i}}
\newcommand{\jjj}{\mathtt{j}}
\newcommand{\kkk}{\mathtt{k}}

\newcommand{\dime}{\textrm{dim}_\psi}
\newcommand{\dimeyla}{\textrm{dim}_{\overline{\psi}}}
\newcommand{\dimeala}{\textrm{dim}_{\underline{\psi}}}
\newcommand{\dimh}{\textrm{dim}_H}

\newcommand{\as}{\underline{s}}
\newcommand{\ys}{\overline{s}}
\newcommand{\apsi}{\underline{\psi}}
\newcommand{\ypsi}{\overline{\psi}}

\newcommand{\pallo}{\circ}
\newcommand{\pois}{\backslash}

\newcommand{\eps}{\varepsilon}

\newcommand{\khii}{\text{\lower -.4ex\hbox{$\chi$}}}
\renewcommand{\emptyset}{\textrm{\O}}

\newcommand{\fii}{\varphi}

\newcommand{\roo}{\varrho}

\begin{document}

\title[On natural invariant measures on generalised
IFS]{On natural invariant measures
on generalised iterated function systems}

\author{Antti K\"aenm\"aki}
\address{Department of Mathematics and Statistics \\
         P.O. Box 35 (MaD) \\
         FIN--40014 University of Jyv\"askyl\"a \\
         Finland}
\email{antakae@maths.jyu.fi\newpage\quad \thispagestyle{empty}}
\newpage\quad \thispagestyle{empty}

\thanks{The author is supported by the Academy of Finland, projects
  23795 and 53292}
\date{\today}

\begin{abstract}
We consider the limit set of generalised iterated function systems. Under
the assumption of a natural potential, the so called
cylinder function, we prove the existence of the
invariant probability measure satisfying the equilibrium state.
We motivate this approach by showing that for typical
self--affine sets there exists an ergodic invariant measure having the
same Hausdorff dimension as the set itself.
\end{abstract}

\maketitle

\section{Introduction}

It is well known that applying methods of thermodynamical formalism, we
can find ergodic invariant measures on self--similar and
self--conformal sets satisfying the equilibrium state and having the
same Hausdorff dimension as the set itself. See, for example, Bowen
\cite{bo}, Hutchinson \cite{hu} and Mauldin and Urba\'nski
\cite{mu}. In this work we try to generalise this concept. Our main
objective is to study iterated function systems (IFS) even though we
develop our theory in a more general setting.

We introduce the definition of a cylinder function, which is a crucial
tool in developing the corresponding concept of thermodynamical
formalism for our setting. The use of the cylinder function provides
us a sufficiently general framework to study iterated function systems.
We could also use the notation of subadditive thermodynamical
formalism like in Falconer \cite{fa1}, \cite{fa2} and Barreira
\cite{ba}, but we feel that in studying iterated function systems we
should use more IFS--style notation.
We can think that the idea of the cylinder function is to generalise
the mass distribution, which is well explained in Falconer \cite{fa3}.
Falconer proved in \cite{fa1} that for each approximative equilibrium
state there exists an approximative equilibrium measure, that is,
there is a $k$--invariant measure for which the approximative
topological pressure equals to the sum of the corresponding entropy
and energy. More
precisely, using the notation of this work, for each $t \ge 0$
there exists a Borel probability measure $\mu_k$ such that
\begin{equation}
  \tfrac{1}{k}P^k(t) = \tfrac{1}{k}h_{\mu_k}^k +
                       \tfrac{1}{k}E_{\mu_k}^k(t).
\end{equation}
Letting now $k \to \infty$, the approximative equilibrium state
converges to the desired equilibrium state, but unfortunately we
will lose the invariance. However, Barreira \cite{ba} showed that the
desired equilibrium state can be attained as a supremum, that is,
\begin{equation} \label{eq:suppi}
  P(t) = \sup\bigl( h_\mu + E_\mu(t) \bigr),
\end{equation}
where the supremum is taken over all invariant Borel regular probability
measures. Using the concept of generalised subadditivity, we show that
it is possible to attain the supremum in (\ref{eq:suppi}). We also
prove that this equilibrium measure is ergodic.

We start developing our theory in the symbol space and after
proving the existence of the equilibrium measure, we begin to consider
the geometric projections of the symbol space and the equilibrium
measure. The use of the cylinder function provides us with a significant
generality in producing equilibrium measures for different kind of
settings. A natural question now is: What can we say about the Hausdorff
dimension of the projected symbol space, the so called limit set? To
answer this question we have to assume something on our geometric
projection. We use the concept of an iterated function system for
getting better control of cylinder sets, the sets defining the
geometric projection. To be able to approximate the size of the limit
set, we also need some kind of separation condition
for cylinder sets to avoid too much overlapping among these sets.
Several separation conditions are introduced and relationships
between them are studied in detail. We also study a couple of
concrete examples, namely the similitude IFS, the conformal IFS and
the affine IFS, and we look how our theory turns out in these
particular cases. As an easy consequence we notice that the Hausdorff
dimension of equilibrium measures of the similitude IFS and the
conformal IFS equals to the Hausdorff dimension of the corresponding limit
sets, the self--similar set and the self--conformal set. After proving
the ergodicity and
studying dimensions of the equilibrium measure in our more
general setting, we obtain the same information for ``almost all''
affine IFS's by applying Falconer's result for the Hausdorff dimension of
self--affine sets. This gives a partially positive answer to the
open question proposed by Kenyon and Peres \cite{kp}.

Before going into more detailed preliminaries, let us fix some notation.
As usual, let $I$ be a finite set with at least two elements.
Put $I^* = \bigcup_{n=1}^\infty I^n$ and $I^\infty = I^\N = \{
(i_1,i_2,\ldots) : i_j \in I \text{ for } j \in \N \}$. Thus, if
$\iii \in I^*$, there is $k \in \N$ such that $\iii =
(i_1,\ldots,i_k)$, where $i_j \in I$ for all $j=1,\ldots,k$. We
call this $k$ the \emph{length} of $\iii$ and we denote
$|\iii|=k$. If $\jjj \in I^* \cup I^\infty$, then with the
notation $\iii,\jjj$ we mean the element obtained by juxtaposing
the terms of $\iii$ and $\jjj$. If $\iii \in
I^\infty$, we denote $|\iii|=\infty$, and for
$\iii \in I^* \cup I^\infty$ we put $\iii|_k = (i_1,\ldots,i_k)$
whenever $1 \le k < |\iii|$. We define $[\iii;A] = \{ \iii,\jjj : \jjj
\in A \}$ as $\iii \in I^*$ and $A \subset I^\infty$ and we call the set
$[\iii] = [\iii,I^\infty]$ the \emph{cylinder set of level $|\iii|$}.
We say that two elements $\iii,\jjj \in I^*$ are \emph{incomparable}
if $[\iii] \cap [\jjj] = \emptyset$. Furthermore,
we call a set $A \subset I^*$ incomparable if all its elements
are mutually incomparable. For example, the sets $I$ and $\{
(i_1,i_2),(i_1,i_1,i_2) \}$, where $i_1 \ne i_2$, are incomparable
subsets of $I^*$.

Define
\begin{equation}
  |\iii - \jjj| =
  \begin{cases}
    2^{-\min\{ k-1\; :\; \iii|_k \ne \jjj|_k \}}, \quad &\iii \ne \jjj \\
    0, \quad &\iii = \jjj
  \end{cases}
\end{equation}
whenever $\iii,\jjj \in I^\infty$. Then the couple $(I^\infty,| \cdot |)$ is
a compact metric space. Let us call $(I^\infty,| \cdot |)$ a
\emph{symbol space} and an element $\iii \in I^\infty$ a
\emph{symbol}. If there is no danger of misunderstanding, let us call
also an element $\iii \in I^*$ a symbol. Define the \emph{left shift}
$\sigma : I^\infty \to I^\infty$ by setting
\begin{equation}
 \sigma(i_1,i_2,\ldots) = (i_2,i_3,\ldots).
\end{equation}
Clearly $\sigma$ is continuous and surjective. If $\iii \in I^n$ for
some $n \in \N$, then with the notation $\sigma(\iii)$ we mean the symbol
$(i_2,\ldots,i_n) \in I^{n-1}$. Sometimes, without mentioning it
explicitly, we work also with ``empty symbols'', that is, symbols with
zero length.

For each cylinder we define a cylinder function $\psi_\iii^t :
I^\infty \to (0,\infty)$ depending also on a given parameter $t \ge
0$. The exact definition is introduced at the beginning of the second
chapter. To follow this introduction, the reader is encouraged to keep in
mind the idea of the mass distribution. With the help of the cylinder
function we define a \emph{topological pressure} $P : [0,\infty) \to
\R$ by setting
\begin{equation}
  P(t) = \lim_{n\to\infty} \tfrac{1}{n} \log \sum_{\iii \in I^n}
  \psi_\iii^t(\hhh),
\end{equation}
where $\hhh \in I^\infty$ is some fixed point. Denoting with
$\MM_\sigma(I^\infty)$ the collection of all Borel regular probability
measures on $I^\infty$ which are \emph{invariant}, that is,
$\mu([\iii]) = \sum_{i \in I} \mu([i,\iii])$ for every $\iii \in I^*$,
we define an \emph{energy} $E_\mu : [0,\infty) \to \R$ by setting
\begin{equation}
  E_\mu(t) = \lim_{n \to \infty} \tfrac{1}{n} \sum_{\iii \in I^n}
  \mu([\iii])\log\psi_\iii^t(\hhh)
\end{equation}
and an \emph{entropy} $h_\mu$ by setting
\begin{equation}
  h_\mu = -\lim_{n \to \infty} \tfrac{1}{n} \sum_{\iii \in I^n}
  \mu([\iii])\log\mu([\iii]).
\end{equation}
For the motivation of these definitions, see, for example, Mauldin and
Urba\'nski \cite{mu} and Falconer \cite{fa4}.
For every $\mu \in \MM_\sigma(I^\infty)$ we have $P(t) \ge h_\mu +
E_\mu(t)$, and if there exists a measure $\mu \in \MM_\sigma(I^\infty)$
for which
\begin{equation}
  P(t) = h_\mu + E_\mu(t),
\end{equation}
we call this measure a \emph{$t$--equilibrium measure}. Using the
generalised subadditivity, we will prove the existence of the
$t$--equilibrium measure. We obtain the ergodicity of that measure
essentially because $\mu \mapsto h_\mu + E_\mu(t)$ is an affine
mapping from a convex set whose extreme points are ergodic and
then recalling Choquet's theorem. Applying now
Kingman's subadditive ergodic theorem and the theorem of
Shannon--McMillan, we notice that
\begin{equation}
  P(t) = \lim_{n \to \infty} \tfrac{1}{n} \log
  \frac{\psi^t_{\iii|_n}(\hhh)}{\mu([\iii|_n])}
\end{equation}
for $\mu$--almost all $\iii \in I^\infty$ as $\mu$ is the
$t$--equilibrium measure. Following the ideas of Falconer \cite{fa1},
we introduce an equilibrium dimension $\dime$ for which
$\dime(I^\infty) = t$ exactly when $P(t)=0$. Using the ergodicity, we
will also prove that $\dime(A) = t$ if $P(t)=0$ and $\mu(A)=1$, where
$\mu$ is the $t$--equilibrium measure. In other words, the equilibrium
measure $\mu$ is ergodic, invariant and has full equilibrium dimension.

To project this setting into $\R^d$ we need some kind of geometric
projection. With the geometric projection here we mean mappings
obtained by the
following construction. Let $X \subset \R^d$ be a compact set with
nonempty interior. Choose then a collection $\{ X_\iii : \iii \in I^*
\}$ of nonempty closed subsets of $X$ satisfying
\begin{itemize}
  \item[(1)] $X_{\iii,i} \subset X_\iii$ for every $\iii \in I^*$ and
  $i \in I$,
  \item[(2)] $d(X_\iii) \to 0$, as $|\iii| \to \infty$.
\end{itemize}
Here $d$ means the diameter of a given set. We define a
\emph{projection mapping} to be the function $\pi : I^\infty \to X$,
for which
\begin{equation}
  \{ \pi(\iii) \} = \bigcap_{n=1}^\infty X_{\iii|_n}
\end{equation}
as $\iii \in I^\infty$. The compact set $E = \pi(I^\infty)$ is called a
\emph{limit set}, and if there is no danger of misunderstanding, we call
also the sets $\pi([\iii])$, where $\iii \in I^*$, cylinder
sets. In general, it is really hard to study the geometric properties
of the limit set, for example, to determine
the Hausdorff dimension. We might come
up against the following problems: There is too much overlapping among
the cylinder sets and it is too difficult to approximate the size of
these sets. Therefore we introduce geometrically stable IFS's.
With the \emph{iterated function system (IFS)} we mean the collection
$\{ \fii_i : i \in I \}$ of contractive injections from $\Omega$ to
$\Omega$, for which $\fii_i(X) \subset X$ as $i \in I$. Here $\Omega
\supset X$ is an open subset of $\R^d$. We set
$X_\iii = \fii_\iii(X)$, where $\fii_\iii = \fii_{i_1} \pallo \cdots
\pallo \fii_{i_{|\iii|}}$ as $\iii \in I^*$, and making now a suitable
choice for the mappings $\fii_i$, we can have the limit set $E$ to be
a self--similar set or a self--affine set, for example. Likewise,
changing the choice of the cylinder function, we can have the equilibrium
measure $\mu$ to have different kind of properties, and thus, making a
suitable choice, the measure $m = \mu \pallo \pi^{-1}$ might be useful
in studying the geometric properties of the limit set. If there is no
danger of misunderstanding, we call also the projected equilibrium
measure $m$ an equilibrium measure.
We say that IFS is \emph{geometrically stable} if it satisfies a
bounded overlapping condition and the mappings of IFS satisfy the following
bi--Lipschitz condition: for each $\iii \in I^*$ there exist
constants $0<\as_\iii<\ys_\iii<1$ such that
\begin{equation}
  \as_\iii |x-y| \le |\fii_\iii(x) - \fii_\iii(y)| \le \ys_\iii |x-y|
\end{equation}
for every $x,y \in \Omega$. The exact definition of these constants is
introduced in Chapter 3. To follow this introduction the reader
can think for simplificity that for each $i \in I$ there exist such
constants and
$\as_\iii = \as_{i_1} \cdots \as_{i_{|\iii|}}$ and $\ys_\iii = \ys_{i_1}
\cdots \ys_{i_{|\iii|}}$ as $\iii \in I^*$. The upper and lower
bounds of the bi--Lipschitz condition are crucial for getting upper
and lower bounds for the size of the cylinder sets. The \emph{bounded
overlapping} is satisfied if the cardinality of the set
$\{ \iii \in I^* : \fii_\iii(X) \cap B(x,r) \ne \emptyset \text{ and }
\as_\iii < r \le \as_{\iii|_{|\iii|-1}} \}$ is uniformly bounded as $x
\in X$ and $0<r<r_0=r_0(x)$.

The class of
geometrically stable IFS's includes many interesting cases of IFS's,
for example, a conformal IFS satisfying the OSC and the so called boundary
condition and an affine IFS satisfying the SSC. The open set condition
(OSC) and the strong separation condition (SSC) are commonly used
examples of separation conditions we need to use for having not too
much overlapping among the cylinder sets. We prove that for the
Hausdorff dimension of the limit set of geometrically stable
IFS's, there exist natural upper and lower bounds obtained from the
bi--Lipschitz constants. It is now very tempting to guess that for
geometrically stable IFS's, making a good choice for the cylinder
function, it could be possible to have the same equilibrium dimension
and Hausdorff dimension for the limit set, and thus it would be
possible to obtain
the Hausdorff dimension from the behaviour of the topological
pressure. It has been already proved that this is true for similitude
and conformal IFS's and also for ``almost all'' affine IFS's.
Recalling now that the equilibrium measure has full equilibrium
dimension, we conclude that in many cases, like in ``almost all''
affine IFS's, making a good choice for the cylinder function, we can
have an ergodic invariant measure on the limit set having full
Hausdorff dimension.

\begin{ack}
  The author is deeply indebted to Professor Pertti Mattila for his
  valuable comments and suggestions for the manuscript.
\end{ack}

\section{Cylinder function and equilibrium measure}

In this chapter we introduce the definition of the cylinder
function. Using the cylinder function we are able to define tools of
thermodynamical formalism. In this setting we prove the existence of
a so called equilibrium measure.

Take $t \ge 0$ and $\iii \in I^*$. We call a function $\psi_\iii^t :
I^\infty \to (0,\infty)$ a \emph{cylinder function} if it satisfies the
following three conditions:
\begin{itemize}
  \item[(1)] There exists $K_t \ge 1$ not depending on $\iii$ such that
  \begin{equation}
    \psi_\iii^t(\hhh) \le K_t \psi_\iii^t(\jjj)
  \end{equation}
  for any $\hhh,\jjj \in I^\infty$.
  \item[(2)] For every $\hhh \in I^\infty$ and integer $1 \le j < |\iii|$
  we have
  \begin{equation}
    \psi_\iii^t(\hhh) \le \psi_{\iii|_j}^t\bigl( \sigma^j(\iii),\hhh \bigr)
	\psi_{\sigma^j(\iii)}^t(\hhh).
  \end{equation}
  \item[(3)] For any given $\delta > 0$ there exist constants $0<\as_\delta<1$
	and $0<\ys_\delta<1$ depending only on $\delta$ such that
  \begin{equation}
    \psi_\iii^t(\hhh)\as_\delta^{|\iii|} \le \psi_\iii^{t+\delta}(\hhh) \le
	\psi_\iii^t(\hhh)\ys_\delta^{|\iii|}
  \end{equation}
  for every $\hhh \in I^\infty$. We assume also that $\as_\delta,\ys_\delta
  \nearrow 1$ as $\delta \searrow 0$ and that $\psi_\iii^0 \equiv 1$.
\end{itemize}

Note that when we speak about one cylinder function, we always assume
there is a collection of them defined for $\iii \in I^*$ and $t>0$.
Let us comment on these conditions. The first one is called the
\emph{bounded variation principle (BVP)} and it says that the value of
$\psi_\iii^t(\hhh)$ cannot vary too much; roughly speaking,
$\psi_\iii^t$ is essentially constant. The second condition is called
the \emph{submultiplicative chain rule for the cylinder function} or just
\emph{subchain rule} for short. If the subchain rule is satisfied
with equality, we call it a \emph{chain rule}. The third condition is
there just to guarantee the nice behaviour of the cylinder function with
respect to the parameter $t$. It also implies that
\begin{equation}
  \as_t^{|\iii|} \le \psi_\iii^t(\hhh) \le \ys_t^{|\iii|}
\end{equation}
with any choice of $\hhh \in I^\infty$.

For each $k \in \N$, $\iii \in I^{k*} := \bigcup_{n=1}^\infty I^{kn}$
and $t \ge 0$ define
a function $\psi_\iii^{t,k} : I^\infty \to (0,\infty)$ by setting
\begin{equation}
  \psi_\iii^{t,k}(\hhh) = \prod_{j=0}^{|\iii|/k - 1}
  \psi_{\sigma^{jk}(\iii)|_k}^t\bigl( \sigma^{(j+1)k}(\iii),\hhh \bigr)
\end{equation}
as $\hhh \in I^\infty$. Clearly, now $\psi_\iii^t(\hhh) \le
\psi_\iii^{t,k}(\hhh)$ for every $k \in \N$ and $\iii \in I^{k*}$
using the subchain rule. Note that if the chain rule is satisfied, then
$\psi_\iii^t(\hhh) = \psi_\iii^{t,k}(\hhh)$ for every $k \in \N$ and
that we always have $\psi_\iii^t(\hhh) = \psi_\iii^{t,|\iii|}(\hhh)$.

It is very tempting to see these functions as cylinder functions
satisfying the chain rule on $I^{k*}$. Indeed, straight from the
definitions we get the chain rule and condition (3) satisfied. However,
to get the BVP for $\psi_\iii^{t,k}$ we need better information on the
local behaviour of the function $\psi_\iii^t$. More precisely, we need
better control over the variation of $\psi_\iii^t$ in small scales. We
call a cylinder function from which we get the BVP for $\psi_\iii^{t,k}$
with any choice of $k \in \N$ \emph{smooth cylinder function}.
We say that a mapping $f : I^\infty \to \R$ is a \emph{Dini function}
if
\begin{equation}
  \int_0^1 \frac{\omega_f(\delta)}{\delta} d\delta < \infty,
\end{equation}
where
\begin{equation}
  \omega_f(\delta) = \sup_{|\iii - \jjj| \le \delta} |f(\iii) - f(\jjj)|
\end{equation}
is the \emph{modulus of continuity}. Observe that H\"older continuous
functions are always Dini.

\begin{proposition} \label{thm:kakspisteyks}
  Suppose the cylinder function is Dini. Then it is smooth and
  functions $\psi_\iii^{t,k}$ are cylinder functions satisfying the
  chain rule on $I^{k*}$.
\end{proposition}

\begin{proof}
  It suffices to verify the BVP.
  For each $k \in \N$ we denote $\omega_k(\delta) = \max_{\iii \in
  I^k} \omega_{\psi_\iii^t}(\delta)$. Using now the assumption and
  the definitions we have for each $\iii \in I^{k*}$
  \begin{align} \label{eq:diniupperbound}
    \log\psi_\iii^{t,k}(\hhh)&-\log\psi_\iii^{t,k}(\jjj) =
    \sum_{j=0}^{|\iii|/k-1} \log\Biggl( \frac{ 
      \psi_{\sigma^{jk}(\iii)|_k}^t 
      \bigl(\sigma^{(j+1)k}(\iii),\hhh \bigr)}
     {\psi_{\sigma^{jk}(\iii)|_k}^t
      \bigl(\sigma^{(j+1)k}(\iii),\jjj \bigr)} \Biggr)  \notag \\ 
    &= \sum_{j=0}^{|\iii|/k-1} \log\Biggl( 1 +
      \frac{\psi_{\sigma^{jk}(\iii)|_k}^t 
      \bigl(\sigma^{(j+1)k}(\iii),\hhh \bigr) - \psi_{\sigma^{jk}(\iii)|_k}^t 
      \bigl(\sigma^{(j+1)k}(\iii),\jjj \bigr)}{\psi_{\sigma^{jk}(\iii)|_k}^t 
      \bigl(\sigma^{(j+1)k}(\iii),\jjj \bigr)}  \Biggr)  \notag \\
    &\le \as_t^{-k} \sum_{j=0}^{|\iii|/k-1} \Bigl| \psi_{\sigma^{jk}(\iii)|_k}^t 
      \bigl(\sigma^{(j+1)k}(\iii),\hhh \bigr) - \psi_{\sigma^{jk}(\iii)|_k}^t 
      \bigl(\sigma^{(j+1)k}(\iii),\jjj \bigr) \Bigr|  \notag \\
    &\le \as_t^{-k} \sum_{j=0}^{|\iii|/k-1} \omega_k\bigl(
      2^{-(|\iii|-(j+1)k)} \bigr) \\
    &\le \as_t^{-k} \int_0^\infty \omega_k\bigl( 2^{-(\eta-1)k} \bigr)
      d\eta \notag \\
    &= \frac{1}{\as_t^k k\log 2} \int_0^1
      \frac{\omega_k(\delta)}{\delta} d\delta, \notag
  \end{align}
  whenever $\hhh,\jjj \in I^\infty$ by substituting $\eta =
  -\tfrac{1}{k}(\log_2 \delta)+1$ and $d\eta = -(\delta k\log 2)^{-1}
  d\delta$. This gives
  \begin{equation}
    \frac{\psi_\iii^{t,k}(\hhh)}{\psi_\iii^{t,k}(\jjj)} \le K_{t,k},
  \end{equation}
  where the logarithm of $K_{t,k}$ equals to the finite upper bound found
  in (\ref{eq:diniupperbound}).
\end{proof}


Of course, a cylinder function satisfying the chain rule is always
smooth, since the BVP for $\psi_\iii^{t,k}$ is satisfied with the
constant $K_t$. Observe that if we have a cylinder function satisfying
the chain rule, but not the BVP, then the previous proposition gives
us a sufficient condition for the BVP to hold, namely the Dini condition.
Next, we introduce an important property of functions of the following type.
We say that a function $a : \N \times \N \cup \{ 0 \} \to \R$ satisfies
the \emph{generalised subadditive condition} if
\begin{equation}
  a(n_1+n_2,0) \le a(n_1,n_2) + a(n_2,0)
\end{equation}
and $|a(n_1,n_2)| \le n_1C$ for some constant $C$. Furthermore, we say
that this function is \emph{subadditive} if in addition $a(n_1,n_2) =
a(n_1,0)$ for all $n_1 \in \N$ and $n_2 \in \N \cup \{ 0 \}$.


\begin{lemma} \label{thm:gensubadd}
  Suppose that a function $a : \N \times \N \cup \{ 0 \} \to \R$
  satisfies the generalised subadditive condition. Then
  \begin{equation}
    \tfrac{1}{n} a(n,0) \le \tfrac{1}{kn} \sum_{j=0}^{n-1} a(k,j) +
    \tfrac{3k}{n}C 
  \end{equation}
  for some constant $C$ whenever $0<k<n$. Moreover, if this function is
  subadditive, then the limit $\lim_{n \to \infty} \tfrac{1}{n}
  a(n,0)$ exists and equals to $\inf_n \tfrac{1}{n} a(n,0)$.
\end{lemma}

\begin{proof}
  We follow the ideas found in Lemma 4.5.2 of Katok and Hasselblatt \cite{ka}.
  Fix $n \in \N$ and choose $0<k<n$. Now for each integer $0\le q <k$
  we define $\alpha(q) = \lfloor \tfrac{n-q-1}{k} \rfloor$ to be the
  integer part of $\tfrac{n-q-1}{k}$. Straight from this definition we
  shall see that $\alpha$ is non--increasing,
  \begin{equation} \label{eq:ekaomin}
    n-k-1 < \alpha(q)k + q \le n-1
  \end{equation}
  and
  \begin{equation} \label{eq:tokaomin}
    \tfrac{n}{k} - 2 < \alpha(q) \le \tfrac{n-1}{k}
  \end{equation}
  whenever $0 \le q < k$. Temporarily fix $q$ and take $0 \le l <
  \alpha(q)$ and $0 \le i < k$. Now
  \begin{equation}
    q-1 < lk+q+i < \alpha(q)k+q
  \end{equation}
  and therefore,
  \begin{equation}
    \{ 0,\ldots,n-1 \} = \{ lk+q+i : 0 \le l < \alpha(q),\; 0\le i<k
                         \} \cup S_q,
  \end{equation}
  where $S_q$ is the union of the sets $S_q^1 = \{ 0,\ldots,q-1 \}$ and
  $S_q^2 = \{ \alpha(q)k+q,\ldots,n-1 \}$. Using (\ref{eq:ekaomin}), we
  notice that $1 \le \#S_q^2 \le k$. It follows from
  (\ref{eq:tokaomin}) that $\alpha(q)$ can attain at maximum two
  values, namely $\lfloor \tfrac{n-1}{k} \rfloor$ and $\lfloor
  \tfrac{n-1}{k} \rfloor - 1$. Let $q_0$ be the largest integer for
  which $\alpha(q_0) = \lfloor \tfrac{n-1}{k} \rfloor$. Then clearly,
  \begin{equation} \label{eq:kolmasomin}
    \{ lk+q : 0 \le l \le \alpha(q),\; 0 \le q < k \} =
    \{ 0,\ldots,\alpha(q_0)k+q_0 \}.
  \end{equation}
  By the choice of $q_0$ it holds also that $\alpha(q_0) =
  (n-q_0-1)/k$ and thus $\alpha(q_0)k + q_0 = n-1$.

  It is clear that $\#S_q^1 = q$. It is also clear that $S_q^2 = \{
  n-k+q,\ldots,n-1 \}$ if $q_0=k-1$. But if not, we notice that
  $\alpha(q_0+1)=\alpha(q_0)-1=(n-q_0-k-1)/k$, and thus $\alpha(q_0+1)k
  + q_0 + 1 = n-k$. Therefore, defining a bijection $\eta$ between sets $\{
  0,\ldots,k-1 \}$ and $\{ 1,\ldots,k \}$ by setting
  \begin{equation}
    \eta(q) =
    \begin{cases}
      q_0-q+1,   \quad &0 \le q \le q_0 \\
      q_0-q+k+1, \quad &q_0 < q < k,
    \end{cases}
  \end{equation}
  we have $\#S_q^2 = \eta(q)$ for all $0 \le q < k$.

  Since $n$ is of the form $\eta(q) + \alpha(q)k + q$ for any $0 \le
  q < k$, we get, using the assumption several times that
  \begin{align}
    a(n,0) &= a\bigl( \eta(q),\alpha(q)k+q \bigr) +
              \sum_{l=1}^{\alpha(q)} a\bigl( k,(\alpha(q)-l)k+q \bigr) +
              a(q,0) \notag \\
           &\le \sum_{l=0}^{\alpha(q)-1} a(k,lk+q) + 2kC \\
           &\le \sum_{l=0}^{\alpha(q)} a(k,lk+q) + 3kC. \notag
  \end{align}
  In fact, we have
  \begin{align} \label{eq:neljasomin}
    \tfrac{1}{n} a(n,0) &\le \tfrac{1}{kn} \sum_{q=0}^{k-1} \Biggl(
    \sum_{l=0}^{\alpha(q)} a(k,lk+q) + 3kC \Biggl) \notag \\
    &= \tfrac{1}{kn} \sum_{j=0}^{n-1} a(k,j) + \tfrac{3k}{n}C
  \end{align}
  using (\ref{eq:kolmasomin}).

  If our function is subadditive, we have
  \begin{equation}
    \limsup_{n \to \infty} \tfrac{1}{n} a(n,0) \le \tfrac{1}{k} a(k,0)
  \end{equation}
  with any choice of $k$ using (\ref{eq:neljasomin}). This also
  finishes the proof.
\end{proof}

Now we define the basic concepts for thermodynamical formalism with the
help of the cylinder function.
Fix some $\hhh \in I^\infty$. We call the following limit
\begin{equation}
  P(t) = \lim_{n \to \infty} \tfrac{1}{n} \log \sum_{\iii \in I^n}
  \psi_\iii^t(\hhh),
\end{equation}
if it exists,
the \emph{topological pressure for the cylinder function} or just
\emph{topological pressure} for short. For each $k \in \N$ we also
denote
\begin{align}
  \overline{P}^k(t)  &= \limsup_{n \to \infty} \tfrac{1}{n} \log
                        \sum_{\iii \in I^{kn}} \psi_\iii^{t,k}(\hhh)
                        \quad \text{and} \notag \\
  \underline{P}^k(t) &= \liminf_{n \to \infty} \tfrac{1}{n} \log
                        \sum_{\iii \in I^{kn}} \psi_\iii^{t,k}(\hhh).
\end{align}
If they agree, we denote the common value with $P^k(t)$.
Recall that the collection of all Borel regular probability measures
on $I^\infty$ is denoted by $\MM(I^\infty)$. Denote
\begin{equation}
  \MM_\sigma(I^\infty) = \{ \mu \in \MM(I^\infty) :
                            \mu \text{ is invariant} \},
\end{equation}
where the invariance of $\mu$ means that $\mu([\iii]) = \mu\bigl(
\sigma^{-1}([\iii]) \bigr)$ for every $\iii \in I^*$.
Now $\MM_\sigma(I^\infty)$ is a 
nonempty closed subset of the compact set $\MM(I^\infty)$ in the weak
topology. For given $\mu \in \MM_\sigma(I^\infty)$ we define an
\emph{energy for the cylinder function} $E_\mu(t)$, or just \emph{energy}
for short, by setting
\begin{equation}
  E_\mu(t) = \lim_{n \to \infty} \tfrac{1}{n} \sum_{\iii \in I^n}
             \mu([\iii]) \log\psi_\iii^t(\hhh)
\end{equation}
provided that the limit exists and an \emph{entropy} $h_\mu$ by setting
\begin{equation}
  h_\mu = \lim_{n \to \infty} \tfrac{1}{n} \sum_{\iii \in I^n}
          H\bigl( \mu([\iii]) \bigr)
\end{equation}
provided that the limit exists,
where $H(x) = -x\log{x}$, as $x > 0$, and $H(0)=0$. Note that $H$ is
concave. For each $k \in \N$ we also denote
\begin{align}
  \overline{E}_\mu^k(t)  &= \limsup_{n \to \infty} \tfrac{1}{n}
                            \sum_{\iii \in I^{kn}} \mu([\iii])
                            \psi_\iii^{t,k}(\hhh) \quad \text{and}
                            \notag \\
  \underline{E}_\mu^k(t) &= \liminf_{n \to \infty} \tfrac{1}{n}
                            \sum_{\iii \in I^{kn}} \mu([\iii])
                            \psi_\iii^{t,k}(\hhh).
\end{align}
If they agree, we denote the common value with $E_\mu^k(t)$. Finally,
we similarly denote
\begin{equation}
  h_\mu^k = \lim_{n \to \infty} \tfrac{1}{n} \sum_{\iii \in I^{kn}}
            H\bigl( \mu([\iii]) \bigr).
\end{equation}
Let us next justify the existence of these limits using the power of
subadditive sequences. We will actually prove a little more than just
subadditivity as we can see from the following lemma.

\begin{lemma} \label{thm:kaksisubjonoa}
  For any given $\mu \in \MM(I^\infty)$ the following functions
  \begin{itemize}
    \item[(1)] $(n_1,n_2) \mapsto \sum_{\iii \in I^{n_1}} H\bigl( \mu \pallo
               \sigma^{-n_2}([\iii]) \bigr)$ \text{ and}
    \item[(2)] $(n_1,n_2) \mapsto \sum_{\iii \in I^{n_1}} \mu \pallo
               \sigma^{-n_2}([\iii]) \log\psi_\iii^t(\hhh) + \log K_t$
  \end{itemize}
  defined on $\N \times \N \cup \{ 0 \}$ satisfy the generalised
  subadditive condition. Furthermore, if $\mu \in \MM_\sigma(I^\infty)$,
  the functions are subadditive.
\end{lemma}

\begin{proof}
  For every $n_1 \in \N$ and $n_2 \in \N \cup \{ 0 \}$ we have
  \begin{align}
    \sum_{\iii \in I^{n_1+n_2}} H\bigl( \mu([\iii]) \bigr) &=
    -\sum_{\iii \in I^{n_1}}\sum_{\jjj \in I^{n_2}} \mu([\jjj,\iii])
    \log\mu([\jjj,\iii]) \notag \\
    &= -\sum_{\iii \in I^{n_1}}\sum_{\jjj \in I^{n_2}}
    \mu([\jjj,\iii]) \log\frac{\mu([\jjj,\iii])}{\mu([\jjj])} -
    \sum_{\iii \in I^{n_1}}\sum_{\jjj \in I^{n_2}} \mu([\jjj,\iii])
    \log\mu([\jjj]) \notag \\
    &= \sum_{\iii \in I^{n_1}}\sum_{\jjj \in I^{n_2}}
    \mu([\jjj]) H\biggl( \frac{\mu([\jjj,\iii])}{\mu([\jjj])} \biggr)
    + \sum_{\jjj \in I^{n_2}} H\bigl( \mu([\jjj]) \bigr) \\
    &\le \sum_{\iii \in I^{n_1}} H\Biggl( \sum_{\jjj \in I^{n_2}}
    \mu([\jjj,\iii]) \Biggr) + \sum_{\jjj \in I^{n_2}} H\bigl(
    \mu([\jjj]) \bigr) \notag
  \end{align}
  using the concavity of the function $H$. Note that while calculating, we can
  sum over only cylinders with positive measure. Using the concavity again,
  we get
  \begin{align}
    \frac{1}{(\#I)^{n_1}} \sum_{\iii \in I^{n_1}} H\Biggl( \sum_{\jjj
    \in I^{n_2}} &\mu([\jjj,\iii]) \Biggr) \le H\Biggl(
    \frac{1}{(\#I)^{n_1}} \sum_{\iii \in I^{n_1+n_2}} \mu([\iii]) \Biggr)
    \notag \\
    &= \frac{1}{(\#I)^{n_1}} \log(\#I)^{n_1},
  \end{align}
  which finishes the proof of (1).

  For every $n_1 \in \N$ and $n_2 \in \N \cup \{ 0 \}$ we have
  \begin{align}
    \sum_{\iii \in I^{n_1+n_2}} \mu([\iii])\log\psi_\iii^t(\hhh) &\le
    \sum_{\iii \in I^{n_1+n_2}}
    \mu([\iii])\log\psi_{\sigma^{n_2}(\iii)}^t(\hhh) \notag \\
    &\qquad + \sum_{\iii \in I^{n_1+n_2}} \mu([\iii])
    \log\psi_{\iii|_{n_2}}^t(\sigma^{n_2}(\iii),\hhh) \notag \\
    &\le \sum_{\iii \in I^{n_1}} \mu \pallo \sigma^{-n_2}([\iii])
    \log\psi_\iii^t(\hhh) \\
    &\qquad + \sum_{\iii \in I^{n_2}} \mu([\iii])
    \log\psi_\iii^t(\hhh) + \log K_t \notag
  \end{align}
  using the BVP and the subchain rule. From the condition (3) of the
  definition of the cylinder function it follows that
  \begin{equation}
    n_1\log \as_t \le \sum_{\iii \in I^{n_1+n_2}} \mu([\iii])
    \log\psi_{\sigma^{n_2}(\iii)}^t(\hhh) \le n_1\log \ys_t,
  \end{equation}
  which finishes the proof of (2).

  The last statement follows directly from the definition of the
  invariant measure.
\end{proof}

Now we can easily conclude the existence of the previously defined
limits. Compare the following proposition also with Chapter 3 of
Falconer \cite{fa2}.

\begin{proposition} \label{thm:perusomin}
  For any given $\mu \in \MM_\sigma(I^\infty)$ it holds that
  \begin{itemize}
    \item[(1)] $P(t)$ exists and equals to $\inf_n\tfrac{1}{n} \bigl(
    \log\sum_{\iii \in I^n} \psi_\iii^t(\hhh) + C_t \bigr)$ with any $C_t
    \ge \log K_t$,
    \item[(2)] $E_\mu(t)$ exists and equals to $\inf_n\tfrac{1}{n}
    \bigl( \sum_{\iii \in I^n} \mu([\iii]) \log\psi_\iii^t(\hhh) + C_t
    \bigr)$ with any $C_t \ge \log K_t$,
    \item[(3)] $h_\mu$ exists and equals to $\inf_n\tfrac{1}{n}
    \sum_{\iii \in I^n} H\bigl( \mu([\iii]) \bigr)$,
    \item[(4)] topological pressure is continuous and strictly
  decreasing and there exists a unique $t \ge 0$ such that $P(t)=0$.
  \end{itemize}
  Furthermore, if the cylinder function is smooth, all the previous
  conditions hold for $P^k(t)$, $E_\mu^k(t)$ and $h_\mu^k$ with any
  given $k \in \N$. It holds also (even without the smoothness
  assumption) that
  \begin{itemize}
    \item[(5)] $P(t) = \lim_{k \to \infty} \tfrac{1}{k}
    \overline{P}^k(t) = \lim_{k \to \infty} \tfrac{1}{k}
    \underline{P}^k(t) = \inf_k \tfrac{1}{k} \overline{P}^k(t)
    = \inf_k \tfrac{1}{k} \underline{P}^k(t)$,
    \item[(6)] $E_\mu(t) = \lim_{k \to \infty} \tfrac{1}{k}
    \overline{E}_\mu^k(t) = \lim_{k \to \infty} \tfrac{1}{k}
    \underline{E}_\mu^k(t) = \inf_k \tfrac{1}{k} \overline{E}_\mu^k(t)
    = \inf_k \tfrac{1}{k} \underline{E}_\mu^k(t)$,
    \item[(7)] $h_\mu = \tfrac{1}{k} h_\mu^k$ for every $k \in \N$.
  \end{itemize}
  Finally, none of these limits depends on the choice of $\hhh \in I^\infty$.
\end{proposition}

\begin{proof}
  Take $\hhh \in I^\infty$ and $\mu \in \MM_\sigma(I^\infty)$. From the
  subchain rule we get
  \begin{align}
    \sum_{\iii \in I^{n_1+n_2}} \psi_\iii^t(\hhh) &\le \sum_{\iii \in
    I^{n_1+n_2}} \psi_{\iii|_{n_1}}^t\bigl( \sigma^{n_1}(\iii),\hhh
    \bigr) \psi_{\sigma^{n_1}(\iii)}^t(\hhh) \notag \\
    &\le K_t \sum_{\iii \in I^{n_1}} \psi_\iii^t(\hhh) \sum_{\iii \in
    I^{n_2}} \psi_\iii^t(\hhh)
  \end{align}
  using the BVP for any choice of $n_1,n_2 \in \N$. Thus, using Lemma
  \ref{thm:gensubadd}, we get (1). Statements (2) and (3) follow
  immediately from the invariance of $\mu$ and Lemmas
  \ref{thm:kaksisubjonoa} and \ref{thm:gensubadd}.

  Using the assumption (3) in the definition of the cylinder function, we have
  for fixed $n \in \N$
  \begin{align}
    \log \as_\delta &+ \tfrac{1}{n} \log \sum_{\iii \in I^n}
    \psi_\iii^t(\hhh) \le \tfrac{1}{n} \log \sum_{\iii \in I^n}
    \psi_\iii^{t+\delta}(\hhh) \notag \\
    &\le \log \ys_\delta + \tfrac{1}{n} \log \sum_{\iii \in I^n}
    \psi_\iii^t(\hhh) 
  \end{align}
  with any choice of $\delta > 0$. Letting $n \to \infty$, we get $0 <
  \log \tfrac{1}{\ys_\delta} \le P(t) - P(t+\delta) \le \log
  \tfrac{1}{\as_\delta}$. This gives the continuity of the topological
  pressure since $\as_\delta,\ys_\delta \nearrow 1$ as $\delta \searrow
  0$. It says also that the topological pressure is strictly
  decreasing and $P(t) \to -\infty$, as $t \to \infty$. Since $P(0) =
  \log\# I$, we have proved (4).

  Assuming the cylinder function to be smooth, we notice that $\psi_\iii^{t,k}$
  are cylinder functions on $I^{k*}$ with any choice of $k \in \N$, and,
  therefore, the previous proofs apply. Using the BVP, we get
  \begin{align}
    \tfrac{1}{kn} \log \sum_{\iii \in I^{kn}} \psi_\iii^{t,k}(\hhh)
    &\le \tfrac{1}{kn} \log K_t^n \sum_{\iii \in I^{kn}}
    \prod_{j=0}^{n-1} \psi_{\sigma^{jk}(\iii)|_k}^t(\hhh) \notag \\
    &= \tfrac{1}{k} \log K_t + \tfrac{1}{kn} \log \Biggl( \sum_{\iii
    \in I^k} \psi_\iii^t(\hhh) \Biggr)^n
  \end{align}
  for any choice of $k,n \in \N$. Therefore, due to the subchain rule,
  \begin{align}
    P(t) &\le \tfrac{1}{kn} \log\sum_{\iii \in I^{kn}}
    \psi_\iii^t(\hhh) + \tfrac{1}{kn} \log K_t \notag \\
    &\le \tfrac{1}{kn} \log\sum_{\iii \in I^{kn}}
    \psi_\iii^{t,k}(\hhh) + \tfrac{1}{kn} \log K_t \\
    &\le \tfrac{1}{k} \log\sum_{\iii \in I^k} \psi_\iii^t(\hhh) +
    \tfrac{1}{k} \log K_t + \tfrac{1}{kn} \log K_t \notag
  \end{align}
  using (1). Now letting $n \to \infty$ and then $k \to \infty$, we get
  (5). Similarly, using the invariance of $\mu$ and the BVP, we have
  \begin{align}
    \tfrac{1}{kn} \sum_{\iii \in I^{kn}} \mu([\iii])
    &\log\psi_\iii^{t,k}(\hhh) \le \tfrac{1}{kn} \sum_{\iii \in I^{kn}}
    \mu([\iii]) \log K_t^n \prod_{j=0}^{n-1}
    \psi_{\sigma^{jk}(\iii)|_k}^t(\hhh) \notag \\
    &= \tfrac{1}{k} \log K_t + \tfrac{1}{kn} \sum_{j=0}^{n-1}
    \sum_{\iii \in I^{kn}} \mu([\iii]) \log
    \psi_{\sigma^{jk}(\iii)|_k}^t(\hhh) \\
    &= \tfrac{1}{k} \log K_t + \tfrac{1}{k} \sum_{\iii \in I^k}
    \mu([\iii]) \log\psi_\iii^t(\hhh) \notag
  \end{align}
  for any choice of $k,n \in \N$. Therefore
  \begin{align}
    E_\mu(t) &\le \tfrac{1}{kn} \sum_{\iii \in I^{kn}} \mu([\iii])
    \log\psi_\iii^t(\hhh) + \tfrac{1}{kn} \log K_t \notag \\
    &\le \tfrac{1}{kn} \sum_{\iii \in I^{kn}} \mu([\iii])
    \log\psi_\iii^{t,k}(\hhh) + \tfrac{1}{kn} \log K_t \\
    &\le \tfrac{1}{k} \sum_{\iii \in I^k} \mu([\iii])
    \log\psi_\iii^t(\hhh) + \tfrac{1}{k} \log K_t + \tfrac{1}{kn} \log
    K_t \notag
  \end{align}
  using (2). Now letting $n \to \infty$ and then $k \to \infty$, we get
  (6). Using the BVP, we get rid of the dependence on the choice of $\hhh
  \in I^\infty$ on these limits. Noting that (7) is trivial, we have
  finished the proof.
\end{proof}

Note that if a cylinder function satisfy the chain rule, we have $P(t)
= \tfrac{1}{k} P^k(t)$ and $E_\mu(t) = \tfrac{1}{k} E_\mu^k(t)$ for
every choice of $k \in \N$ and $\mu \in \MM_\sigma(I^\infty)$.
With these tools of thermodynamical formalism we are now ready to look
for a special invariant 
measure on $I^\infty$, the so called equilibrium measure. If we denote
$\alpha(\iii) = \psi_\iii^t(\hhh)/\sum_{\jjj \in I^{|\iii|}}
\psi_\jjj^t(\hhh)$, as $\iii \in I^*$, we get, using Jensen's
inequality for any $n \in \N$ and $\mu \in \MM(I^\infty)$,
\begin{align} \label{eq:jensen}
  0 &= 1\log 1 = \tfrac{1}{n} H\Biggl( \sum_{\iii \in I^n} \alpha(\iii)
  \frac{\mu([\iii])}{\alpha(\iii)} \Biggr) \ge \tfrac{1}{n} \sum_{\iii \in I^n}
  \alpha(\iii) H\biggl( \frac{\mu([\iii])}{\alpha(\iii)} \biggr)
  \notag \\
  &= \tfrac{1}{n} \sum_{\iii \in I^n} \mu([\iii]) \biggl( -\log\mu([\iii]) +
  \log\psi_\iii^t(\hhh) - \log\sum_{\jjj \in I^n} \psi_\jjj^t(\hhh) \biggr)
\end{align}
with equality if and only if $\mu([\iii]) = C\alpha(\iii)$ for some
constant $C>0$. Thus, in the view of Proposition \ref{thm:perusomin},
\begin{equation}
  P(t) \ge h_\mu + E_\mu(t)
\end{equation}
whenever $\mu \in \MM_\sigma(I^\infty)$. We call a measure $\mu \in
\MM_\sigma(I^\infty)$ as \emph{$t$--equilibrium measure} if it
satisfies an equilibrium state
\begin{equation}
  P(t) = h_\mu + E_\mu(t).
\end{equation}
In other words, the equilibrium measure (or state) is a solution for
a variational equation $P(t) = \sup_{\mu \in \MM_\sigma(I^\infty)} \bigl(
h_\mu + E_\mu(t) \bigr)$.

Define now for each $k \in \N$ a Perron--Frobenius operator $\FF_{t,k}$
by setting
\begin{equation}
  \bigl( \FF_{t,k}(f) \bigr)(\hhh) = \sum_{\iii \in I^k}
  \psi_\iii^{t,k}(\hhh) f(\iii,\hhh)
\end{equation}
for every continuous function $f : I^\infty \to \R$. Using this
operator, we are able to find our equilibrium measure. Assuming $\bigl(
\FF_{t,k}^{n-1}(f) \bigr)(\hhh) = \sum_{\iii \in I^{k(n-1)}}
\psi_\iii^{t,k}(\hhh) f(\iii,\hhh)$, we get inductively, using the chain
rule,
\begin{align}
  \bigl( \FF_{t,k}^n(f) \bigr)(\hhh) &= \bigl( \FF_{t,k}\bigl(
  \FF_{t,k}^{n-1}(f) \bigr) \bigr)(\hhh) \notag \\
  &= \sum_{\iii \in I^k} \psi_\iii^{t,k}(\hhh) \bigl(
  \FF_{t,k}^{n-1}(f) \bigr)(\iii,\hhh) \notag \\
  &= \sum_{\iii \in I^k} \psi_\iii^{t,k}(\hhh) \sum_{\jjj \in
  I^{k(n-1)}} \psi_\jjj^{t,k}(\iii,\hhh) f(\jjj,\iii,\hhh) \\
  &= \sum_{\iii \in I^{kn}} \psi_\iii^{t,k}(\hhh) f(\iii,\hhh). \notag
\end{align}
Let us then denote with $\FF_{t,k}^*$ the dual operator of
$\FF_{t,k}$. Due to the Riesz representation theorem it operates on
$\MM(I^\infty)$. Relying now on the definitions of these operators, we
may find a special measure using a suitable fixed point theorem.
If the chain rule is satisfied, this is a known result. For example,
see Theorem 1.7 of Bowen \cite{bo}, Theorem 3 of Sullivan \cite{su}
and Theorem 3.5 of Mauldin and Urba\'nski \cite{mu}.

\begin{theorem} \label{thm:aakolme}
  For each $t \ge 0$ and $k \in \N$ there exists a measure $\nu_k \in
  \MM(I^\infty)$ such that
  \begin{equation}
    \nu_k([\iii;A]) = \Pi_k^{-|\iii|/k} \int_A
    \psi_\iii^{t,k}(\hhh) d\nu_k(\hhh),
  \end{equation}
  where $\Pi_k > 0$, $\iii \in I^{k*}$ and $A \subset I^\infty$ is a
  Borel set. Moreover, $\lim_{k \to \infty} \Pi_k^{1/k} =  e^{P(t)}$ and
  if the cylinder function is smooth, $\Pi_k = e^{P^k(t)}$ for every $k
  \in \N$.
\end{theorem}

\begin{proof}
  For fixed $t \ge 0$ and $k \in \N$ define $\Lambda : \MM(I^\infty)
  \to \MM(I^\infty)$ by setting
  \begin{equation}
    \Lambda(\mu) = \frac{1}{\bigl( \FF_{t,k}^*(\mu) \bigr)(I^\infty)}
    \FF_{t,k}^*(\mu).
  \end{equation}
  Take now an arbitrary converging sequence, say, $(\mu_n)$ for which
  $\mu_n \to \mu$ in the weak topology with some $\mu \in
  \MM(I^\infty)$. Then for each continuous $f$ we have
  \begin{align}
    \bigl( \FF_{t,k}^*&(\mu_n) \bigr)(f) = \mu_n\bigl( \FF_{t,k}(f)
    \bigr) \notag \\
    &\to \mu\bigl( \FF_{t,k}(f) \bigr) = \bigl( \FF_{t,k}^*(\mu)
    \bigr)(f)
  \end{align}
  as $n \to \infty$. Thus $\Lambda$ is continuous. Now the
  Schauder--Tychonoff fixed point theorem applies and we find $\nu_k
  \in \MM(I^\infty)$ such that $\Lambda(\nu_k) = \nu_k$. Denoting $\Pi_k
  = \bigl( \FF_{t,k}^*(\nu_k) \bigr)(I^\infty)$, we have
  $\FF_{t,k}^*(\nu_k) = \Pi_k\nu_k$. Take now some Borel set
  $A \subset I^\infty$ and $\iii \in I^{k*}$. Then
  \begin{align} \label{eq:semikonforminen}
    \Pi_k^{|\iii|/k} \nu_k([\iii;A]) &= \bigl(
    (\FF_{t,k}^*)^{|\iii|/k}(\nu_k) \bigr)([\iii;A]) = \nu_k\bigl(
    \FF_{t,k}^{|\iii|/k}(\khii_{[\iii;A]}) \bigr) \notag \\
    &= \int_{I^\infty} \sum_{\jjj \in I^{|\iii|}}
    \psi_\jjj^{t,k}(\hhh) \khii_{[\iii;A]}(\jjj,\hhh) d\nu_k(\hhh) \notag \\
    &= \int_{I^\infty} \psi_\iii^{t,k}(\hhh) \khii_A(\hhh) d\nu_k(\hhh) \\
    &= \int_A \psi_\iii^{t,k}(\hhh) d\nu_k(\hhh), \notag
  \end{align}
  which proves the first claim. It also follows applying the BVP that for
  each $n \in \N$
  \begin{align} \label{eq:semikonfarvio}
    \Pi_k^n = \Pi_k^n &\sum_{\iii \in I^{kn}} \nu_k([\iii]) =
    \int_{I^\infty} \sum_{\iii \in I^{kn}} \psi_\iii^{t,k}(\hhh)
    d\nu_k(\hhh) \notag \\
    &\le K_t^n \sum_{\iii \in I^{kn}} \psi_\iii^{t,k}(\hhh)
  \end{align}
  and, similarly, the other way around.
  Taking now logarithms, dividing by $kn$ and taking the limit, we have
  for each $k \in \N$
  \begin{equation}
    \tfrac{1}{k}\underline{P}^k(t) - \tfrac{1}{k} \log K_t \le
    \tfrac{1}{k} \log\Pi_k \le \tfrac{1}{k}\overline{P}^k(t) +
    \tfrac{1}{k} \log K_t.
  \end{equation}
  If the cylinder function is
  smooth, then for each $k$ there exists a constant $K_{t,k} \ge 1$ for
  which $\psi_\iii^{t,k}(\hhh) \le K_{t,k} \psi_\iii^{t,k}(\jjj)$
  whenever $\hhh,\jjj \in I^\infty$ and $\iii \in I^{k*}$. Using
  this in (\ref{eq:semikonfarvio}), we have finished the proof.
\end{proof}

Note that if a cylinder function satisfies the chain rule, then $\nu_k = \nu$
for every $k \in \N$, where
\begin{equation} \label{eq:conformaldef}
  \nu([\iii;A]) = e^{-|\iii|P(t)} \int_A \psi_\iii^t(\hhh) d\nu(\hhh)
\end{equation}
as $\iii \in I^*$ and $A \subset I^\infty$ is a Borel set. The
measure $\nu$ is called a \emph{$t$--conformal measure}.

\begin{theorem} \label{thm:equilibrium}
  There exists an equilibrium measure.
\end{theorem}

\begin{proof}
  According to Theorem \ref{thm:aakolme}, we have for each $n \in \N$
  a measure $\nu_n \in \MM(I^\infty)$ for which
  \begin{equation} \label{eq:melkeinconf}
    \nu_n([\iii]) = \Pi_n^{-1} \int_{I^\infty} \psi_\iii^t(\hhh)
    d\nu_n(\hhh),
  \end{equation}
  where $\iii \in I^n$ and $\lim_{n \to \infty} \tfrac{1}{n} \log
  \Pi_n = P(t)$. Hence, using the BVP, we get
  \begin{align} \label{eq:kuustahti}
    \tfrac{1}{n} \sum_{\iii \in I^n} \nu_n([\iii]) \bigl(
    &-\log\nu_n([\iii]) + \log\psi_\iii^t(\hhh) \bigr) \notag \\
    &= \tfrac{1}{n} \sum_{\iii \in I^n} \nu_n([\iii]) \biggl( -\log
    \Pi_n^{-1} \int_{I^\infty} \psi_\iii^t(\hhh) d\nu_n(\hhh) + \log
    \psi_\iii^t(\hhh) \biggr) \notag \\
    &\ge \tfrac{1}{n} \sum_{\iii \in I^n} \nu_n([\iii]) (\log \Pi_n -
    \log K_t) \\
    &= \tfrac{1}{n} \log \Pi_n - \tfrac{1}{n} \log K_t \notag
  \end{align}
  for every $n \in \N$. Define now for each $n \in \N$ a probability
  measure
  \begin{equation}
    \mu_n = \tfrac{1}{n} \sum_{j=0}^{n-1} \nu_n \pallo \sigma^{-j}
  \end{equation}
  and take $\mu$ to be some accumulation point of the
  set $\{ \mu_n \}_{n \in \N}$ in the weak topology. Now for any $\iii \in
  I^*$ we have
  \begin{align}
    \bigl|\mu_n([\iii]) - \mu_n\bigl( \sigma^{-1}([\iii]) \bigr)\bigr| &=
    \tfrac{1}{n} \bigl|\nu_n([\iii]) - \nu_n\pallo\sigma^{-n}([\iii])\bigr|
    \notag \\
    &\le \tfrac{1}{n} \to 0,
  \end{align}
  as $n \to \infty$. Thus $\mu \in \MM_\sigma(I^\infty)$. According to
  Lemma \ref{thm:gensubadd} and Proposition \ref{thm:kaksisubjonoa}(1),
  we have, using concavity of $H$,
  \begin{align} \label{eq:seiskatahti}
    \tfrac{1}{n} \sum_{\iii \in I^n} H\bigl( \nu_n([\iii]) \bigr) &\le
    \tfrac{1}{kn} \sum_{j=0}^{n-1} \sum_{\iii \in I^k} H\bigl( \nu_n
    \pallo \sigma^{-j}([\iii]) \bigr) + \tfrac{3k}{n}C_1 \notag \\
    &\le \tfrac{1}{k} \sum_{\iii \in I^k} H\bigl( \mu_n([\iii]) \bigr)
    + \tfrac{3k}{n}C_1
  \end{align}
  for some constant $C_1$ whenever $0<k<n$. Using then Lemma
  \ref{thm:gensubadd} and Proposition \ref{thm:kaksisubjonoa}(2), we
  get
  \begin{align} \label{eq:kasitahti}
    \tfrac{1}{n} \sum_{\iii \in I^n} \nu_n([\iii]) &\log
    \psi_\iii^t(\hhh) + \tfrac{1}{n} \log K_t \notag \\
    &\le \tfrac{1}{kn} \sum_{j=0}^{n-1} \Biggl( \sum_{\iii \in I^k}
    \nu_n \pallo \sigma^{-j}([\iii]) \log \psi_\iii^t(\hhh) + \log K_t
    \Biggr) + \tfrac{3k}{n}C_2 \\
    &= \tfrac{1}{k} \sum_{\iii \in I^k} \mu_n([\iii]) \log
    \psi_\iii^t(\hhh) + \tfrac{1}{k}\log K_t + \tfrac{3k}{n}C_2 \notag
  \end{align}
  for some constant $C_2$ whenever $0<k<n$. Now putting
  (\ref{eq:kuustahti}), (\ref{eq:seiskatahti}) and
  (\ref{eq:kasitahti}) together, we have
  \begin{align}
    \tfrac{1}{n} \log \Pi_n &\le \tfrac{1}{n} \sum_{\iii \in I^n}
    H\bigl( \nu_n([\iii]) \bigr) + \tfrac{1}{n} \sum_{\iii \in I^n}
    \nu_n([\iii]) \log \psi_\iii^t(\hhh) + \tfrac{1}{n} \log K_t
    \notag \\
    &\le \tfrac{1}{k} \sum_{\iii \in I^k} H\bigl( \mu_n([\iii]) \bigr)
    + \tfrac{1}{k} \sum_{\iii \in I^k} \mu_n([\iii]) \log
    \psi_\iii^t(\hhh) \\
    &\quad + \tfrac{3k}{n}C_1 + \tfrac{3k}{n}C_2 + \tfrac{1}{k} \log K_t \notag
  \end{align}
  whenever $0<k<n$. Letting now $n \to \infty$, we get
  \begin{equation}
    P(t) \le \tfrac{1}{k} \sum_{\iii \in I^k} H\bigl( \mu([\iii])
    \bigr) + \tfrac{1}{k} \sum_{\iii \in I^k} \mu([\iii]) \log
    \psi_\iii^t(\hhh) + \tfrac{1}{k} \log K_t
  \end{equation}
  since cylinder sets have empty boundary. The proof is finished by
  letting $k \to \infty$.
\end{proof}

\begin{remark}
  In order to prove the existence of the equilibrium measure, the use
  of the Perron--Frobenius operator is not necessarily needed. Indeed,
  for fixed $\hhh \in I^\infty$ we could define for each $n \in \N$ a
  probability measure
  \begin{equation}
    \nu_n = \frac{\sum_{\iii \in I^n} \psi_\iii^t(\hhh)
    \delta_{\iii,\hhh}}{\sum_{\iii \in I^n} \psi_\iii^t(\hhh)},
  \end{equation}
  where $\delta_\hhh$ is a probability measure with support $\{ \hhh
  \}$. Now with this measure we have equality in (\ref{eq:jensen}),
  which is going to be our replacement for (\ref{eq:kuustahti}) in the
  proof of Theorem \ref{thm:equilibrium}.
\end{remark}

Notice that in the simplest case, where the cylinder function is
constant and satisfies the chain rule, the conformal measure equals to
the equilibrium measure. This can be easily derived from the following
theorem. Compare it also with Theorem 3.8 of Mauldin and Urba\'nski
\cite{mu}.

\begin{theorem} \label{thm:equivalent}
  Suppose the cylinder function satisfies the chain rule. Then
  \begin{equation}
    K_t^{-1} \nu(A) \le \mu(A) \le K_t \nu(A)
  \end{equation}
  for every Borel set $A \subset I^\infty$, where $\nu$ is a $t$--conformal
  measure and $\mu$ is the $t$--equilibrium measure found in Theorem
  \ref{thm:equilibrium}.
\end{theorem}

\begin{proof}
  Using the BVP, we derive from (\ref{eq:conformaldef})
  \begin{align}
    1 &= \sum_{\iii \in I^n} \nu([\iii]) = e^{-nP(t)} \sum_{\iii \in
    I^n} \int_{I^\infty} \psi_\iii^t(\hhh) d\nu(\hhh) \notag \\
    &\le K_t e^{-nP(t)} \sum_{\iii \in I^n} \psi_\iii^t(\hhh)
  \end{align}
  for all $n \in \N$ and, similarly, the other way around. Thus we have
  \begin{equation} \label{eq:summaarvio}
    K_t^{-1} e^{nP(t)} \le \sum_{\iii \in I^n} \psi_\iii^t(\hhh) \le
    K_t e^{nP(t)}
  \end{equation}
  for all $n \in \N$. Note that in view of the chain rule we have
  \begin{align}
    \mu([\iii]) &= \lim_{n \to \infty} \tfrac{1}{n} \sum_{j=0}^{n-1}
    \nu \pallo \sigma^{-j}([\iii]) = \lim_{n \to \infty} \tfrac{1}{n}
    \sum_{j=0}^{n-1} \sum_{\jjj \in I^j} \nu([\jjj,\iii]) \notag \\
    &= \lim_{n \to \infty} \tfrac{1}{n} \sum_{j=0}^{n-1} \sum_{\jjj
    \in I^j} e^{-|\jjj,\iii|P(t)} \int_{I^\infty}
    \psi_{\jjj,\iii}^t(\hhh)d\nu(\hhh) \\
    &= \lim_{n \to \infty} \tfrac{1}{n} \sum_{j=0}^{n-1}
    e^{-(j+|\iii|)P(t)} \int_{I^\infty} \psi_\iii^t(\hhh) \sum_{\jjj
    \in I^j} \psi_\jjj^t(\iii,\hhh) d\nu(\hhh) \notag
  \end{align}
  whenever $\iii \in I^*$ since cylinder sets have empty boundary. Now,
  using (\ref{eq:summaarvio}), we get
  \begin{equation}
    K_t^{-1} \nu([\iii]) \le \mu([\iii]) \le K_t \nu([\iii])
  \end{equation}
  for every $\iii \in I^*$. Pick a closed set $C \subset I^\infty$ and
  define $C_n = \{  \iii \in I^n : [\iii] \cap C \ne \emptyset \}$
  whenever $n \in \N$. Now sets $\bigcup_{\iii \in C_n} [\iii] \supset
  C$ are decreasing as $n=1,2,\ldots$, and, therefore,
  $\bigcap_{n=1}^\infty \bigcup_{\iii \in C_n} [\iii] = C$. Thus,
  \begin{align}
    K_t^{-1} \nu(C) &= K_t^{-1} \lim_{n \to \infty} \sum_{\iii \in
    C_n} \nu([\iii]) \le \lim_{n \to \infty} \sum_{\iii \in C_n}
    \mu([\iii]) \notag \\
    &= \mu(C) \le K_t \nu(C).
  \end{align}
  Let $A \subset I^\infty$ be a Borel set. Then, by the Borel regularity of
  these measures, we may find closed sets $C_1,C_2 \subset A$ such that
  $\nu(C_1\pois A) < \eps$ and $\mu(C_2\pois A) < \eps$ for any given
  $\eps > 0$. Therefore, $\nu(A) \le \nu(C_1) + \eps \le K_t\mu(A) +
  \eps$ and $\mu(A) \le \mu(C_2) + \eps \le K_t\mu(A) + \eps$. Letting
  now $\eps \searrow 0$, we have finished the proof.
\end{proof}

\section{Equilibrium dimension and iterated function system}

In the previous chapter, with the help of the simple structured symbol
space using the cylinder function, we found measures with desired
properties. In the following we will project this situation into
$\R^d$. The natural question now is: What can we say about the
Hausdorff dimension of the projected symbol space, the so called limit
set? To answer this question, we have to make several extra
assumptions, namely, we define the concept of the iterated function system
and we introduce a couple of separation conditions. To illustrate our
theory, we give concrete examples at the end of this chapter.

For fixed $t \ge 0$ we denote with $\mu_t$ a corresponding equilibrium
measure. We define for each $n \in \N$
\begin{equation} \label{eq:falconer_measure1}
  \GG_n^t(A) = \inf\Biggl\{ \sum_{j=1}^\infty \int_{I^\infty}
  \psi_{\iii_j}^t(\hhh) d\mu_t(\hhh) : A \subset \bigcup_{j=1}^\infty
  \,[\iii_j],\; |\iii_j| \ge n \Biggr\}
\end{equation}
whenever $A \subset I^\infty$. Assumptions in Carath\'eodory's
construction (for example, see Chapter 4 of \cite{ma2}) are now
satisfied and we have a Borel regular measure $\GG^t$ on $I^\infty$
with
\begin{equation} \label{eq:falconer_measure2}
  \GG^t(A) = \lim_{n \to \infty} \GG_n^t(A).
\end{equation}

\begin{lemma}
  If $\GG^{t_0}(A) < \infty$, then $\GG^t(A) = 0$ for all $t > t_0$.
\end{lemma}

\begin{proof}
  Let $n \in \N$ and choose a collection of cylinder sets
  $\{[\iii_j]\}_j$ such that $|\iii_j| \ge n$ and $\sum_j
  \int_{I^\infty} \psi_{\iii_j}^{t_0}(\hhh) d\mu_{t_0}(\hhh) \le \GG_n^{t_0}(A)
  + 1$. Then
  \begin{align}
    \GG_n^t(A) &\le \sum_j \int_{I^\infty} \psi_{\iii_j}^t(\hhh)
    d\mu_t(\hhh) \le K_tK_{t_0} \sum_j \int_{I^\infty}
    \psi_{\iii_j}^{t_0}(\hhh)
    d\mu_{t_0}(\hhh) \ys_{t-t_0}^{|\iii_j|} \notag \\
    &\le K_tK_{t_0}\ys_{t-t_0}^n \bigl( \GG_n^{t_0}(A) + 1 \bigr).
  \end{align}
  By letting $n \to \infty$ we have finished the proof.
\end{proof}

Using this lemma, we may now define
\begin{align}
  \dime(A) &= \inf\{ t \ge 0 : \GG^t(A) = 0 \} \notag \\
           &= \sup\{ t \ge 0 : \GG^t(A) = \infty \}
\end{align}
and we call this ``critical value'' the \emph{equilibrium dimension}
of the set $A \subset I^\infty$. Notice that the equilibrium dimension
does not depend on the measure $\mu_t$. In fact, defining the
measure $\GG^t$ by using a fixed $\hhh \in I^\infty$ instead of the
integral average in (\ref{eq:falconer_measure1}), leads us for the
same definition of the equilibrium dimension due to the BVP. The most
important property of the equilibrium dimension is the following
theorem.

\begin{theorem} \label{thm:eqdim}
  $P(t)=0$ if and only if $\dime(I^\infty)=t$.
\end{theorem}

\begin{proof}
  Let us first show that $P(t)<0$ implies $\dime(I^\infty) \le
  t$. Using the BVP, we derive from Theorem \ref{thm:aakolme}
  \begin{align}
    1 &= \sum_{\iii \in I^n} \nu_n([\iii]) = \Pi_n^{-1} \sum_{\iii \in
    I^n} \int_{I^\infty} \psi_\iii^t(\hhh) d\nu_n(\hhh) \notag \\
    &\ge K_t^{-1} \Pi_n^{-1} \sum_{\iii \in I^n} \psi_\iii^t(\hhh),
  \end{align}
  where $\lim_{n \to \infty} \Pi_n^{1/n} = e^{P(t)}$. Now
  \begin{equation}
    \limsup_{n \to \infty} \Biggl( \sum_{\iii \in I^n}
    \psi_\iii^t(\hhh) \Biggr)^{1/n} \le \lim_{n \to \infty}
    (K_t\Pi_n)^{1/n} = e^{P(t)} < 1
  \end{equation}
  and choosing $n_0$ big enough, we have
  \begin{equation}
    \Biggl( \sum_{\iii \in I^n} \psi_\iii^t(\hhh) \Biggr)^{1/n} <
    \frac{1 + e^{P(t)}}{2} < 1
  \end{equation}
  whenever $n \ge n_0$.
  Hence, for any given $\eps > 0$ there exists $n_1 \in \N$ such that
  \begin{equation}
    \sum_{\iii \in I^n} \int_{I^\infty} \psi_\iii^t(\hhh) d\mu(\hhh) < \eps
  \end{equation}
  whenever $n \ge n_1$. This proves the claim.

  For the convenience of the reader, to prove the other direction we
  repeat here the argument of Falconer from \cite{fa1}.
  Let us assume that $t > \dime(I^\infty)$ and $\hhh \in
  I^\infty$. Then, clearly,
  $\GG^t(I^\infty)=0$ and we may choose a finite cover for $I^\infty$ of
  the form $\{ [\iii] : \iii \in A \subset \bigcup_{j=1}^{n_0} I^j
  \}$, where $n_0 \in \N$ is large enough and $A$ is some incomparable
  set such that
  \begin{equation}
    \sum_{\iii \in A} \psi_\iii^t(\hhh) < K_t^{-1}.
  \end{equation}
  Here we can choose a finite cover, since any infinite collection of
  disjoint cylinders will not cover the whole $I^\infty$. Define now
  for each integer $n \ge n_0$ a set
  \begin{align}
    A_n = \{ \iii_1,\ldots,\iii_q \in I^* :\; &\iii_j \in A \text{ as }
             j=1,\ldots,q \text{ with some } q, \notag \\
             &|\iii_1,\ldots,\iii_q| \ge n \text{ and }
             |\iii_1,\ldots,\iii_{q-1}| \le n \}.
  \end{align}
  Now, using the subchain rule, we get with any choice of $\jjj \in I^*$
  \begin{equation}
    \sum_{\iii \in A} \psi_{\jjj,\iii}^t(\hhh) \le K_t
    \psi_\jjj^t(\hhh) \sum_{\iii \in A} \psi_\iii^t(\hhh) \le
    \psi_\jjj^t(\hhh)
  \end{equation}
  whenever $\hhh \in I^\infty$. Thus, inductively, we get for every $n
  \ge n_0$
  \begin{equation}
    \sum_{\iii \in A_n} \psi_\iii^t(\hhh) \le K_t^{-1}.
  \end{equation}
  Assuming $\iii \in I^{n+n_0}$, we have $\iii = \jjj,\kkk$ for some
  $\jjj \in A_n$ and $\kkk \in I^*$ with $|\kkk| \le n_0$. Moreover,
  for each such $\jjj$ there are at most $(\#I)^{n_0}$ such
  $\kkk$. Since $\psi_\iii^t(\hhh) \le \psi_\jjj^t(\kkk,\hhh)
  \ys_t^{|\kkk|} \le \psi_\jjj^t(\kkk,\hhh)$, we have
  \begin{equation}
    \sum_{\iii \in I^{n+n_0}} \psi_\iii^t(\hhh) \le (\#I)^{n_0} K_t
  \sum_{\jjj \in A_n} \psi_\jjj^t(\hhh) \le (\#I)^{n_0}
  \end{equation}
  for all $n \in \N$. From this we derive that $P(t)\le 0$. This also
  finishes the proof.
\end{proof}

So far we have worked only in the symbol space. It has provided us with
a simple structured environment for finding measures with desired
properties. It is, however, more interesting to study geometric
projections of these measures and the symbol space. In the following
we define what we mean by this geometric projection.
Let $X \subset \R^d$ be a compact set with nonempty interior. Choose
then a collection $\{ X_\iii : \iii \in I^* \}$ of nonempty closed
subsets of $X$ satisfying
\begin{itemize}
  \item[(1)] $X_{\iii,i} \subset X_\iii$ for every $\iii \in I^*$ and
  $i \in I$,
  \item[(2)] $d(X_\iii) \to 0$, as $|\iii| \to \infty$.
\end{itemize}
Here $d$ means the diameter of a given set. Define now a
\emph{projection mapping} $\pi : I^\infty \to X$ such that
\begin{equation}
  \{ \pi(\iii) \} = \bigcap_{n=1}^\infty X_{\iii|_n}
\end{equation}
as $\iii \in I^\infty$. It is clear that $\pi$ is continuous. We call
the compact set $E = \pi(I^\infty)$ as the \emph{limit set} of this
collection, and if there is no danger of misunderstanding, we also
call the projected cylinder set a cylinder set.

We could now define a cylinder function for this collection of
sets. But without any additional information the equilibrium dimension has
most likely nothing to do with the Hausdorff dimension of the limit
set. Therefore, in order to determine the Hausdorff dimension, it
is natural to require that the cylinder function somehow represents
the size of the subset $X_\iii$ and also that there is not too much
overlapping among these sets. The use of iterated function systems
with well--chosen mappings and separation condition will provide us
with the sufficient information we need.

Take now $\Omega \supset X$ to be an open subset
of $\R^d$. Let $\{ \fii_\iii : i \in I^* \}$ be a collection of contractive
injections from $\Omega$ to $\Omega$ such that the collection $\{
\fii_\iii(X) : \iii \in I^* \}$ satisfies both properties (1) and (2)
above. By \emph{contractivity} we mean that for every $\iii
\in I^*$ there exists a constant $0< s_\iii <1$ such that
$|\fii_\iii(x)-\fii_\iii(y)| \le
s_\iii |x-y|$ whenever $x,y \in \Omega$. This kind of collection is called
a \emph{general iterated function system}. Furthermore, we call the
collection $\{ \fii_i : i \in I \}$ of the same kind of mappings an
\emph{iterated function system (IFS)}. Defining
$\fii_\iii = \fii_{i_1} \pallo \cdots \pallo
\fii_{i_{|\iii|}}$, as $\iii \in I^*$, we clearly get the
assumptions of general IFS satisfied. In fact, we have $d\bigl(
\fii_\iii(X) \bigr) \le (\max_{i \in I} s_i)^{|\iii|}d(X)$.

To avoid too much overlapping, we need a decent separation condition
for the subsets $\fii_\iii(X)$. We say that a \emph{strong separation
condition (SSC)} is satisfied if $\fii_\iii(X) \cap \fii_\jjj(X) =
\emptyset$ whenever $\iii$ and $\jjj$ are incomparable. For IFS it
suffices to require $\fii_i(X) \cap \fii_j(X) = \emptyset$ for $i \ne
j$. Of course, assuming the SSC would be enough in many cases, but it is a
rather restrictive assumption, and usually we
do not need that much. We say that an \emph{open set condition (OSC)} is
satisfied if $\fii_\iii\bigl( \text{int}(X) \bigr) \cap \fii_\jjj\bigl(
\text{int}(X) \bigr) = \emptyset$ whenever $\iii$ and $\jjj$ are
incomparable. Again, for IFS it suffices to require $\fii_i\bigl(
\text{int}(X) \bigr) \cap \fii_j\bigl(
\text{int}(X) \bigr) = \emptyset$ for $i \ne j$. With the notation
$\text{int}(X)$ we mean the interior of $X$. Furthermore, we say that
a general IFS
has \emph{weak bounded overlapping} if the cardinality of incomparable
subsets of $\{ \iii \in I^* : x \in \fii_\iii(X) \}$
is uniformly bounded as $x \in X$. Trivially, a general IFS satisfying the SSC
has weak bounded overlapping. Assume now that for each $\iii \in I^*$ there
exists a constant $0 < \as_\iii < 1$  such that $\as_\iii \to 0$ as
$|\iii| \to \infty$. Then we say that a general IFS has \emph{bounded
overlapping} if the cardinality of the set $Z(x,r) = \{ \iii \in Z(r)
: \fii_\iii(X) \cap B(x,r) \ne \emptyset \}$ is uniformly bounded as
$x \in X$ and $0<r<r_0=r_0(x)$. Here $Z(r)$ is an incomparable subset
of $\{ \iii \in I^* : \as_\iii < r \le \as_{\iii|_{|\iii|-1}} \}$ such
that $\{ [\iii] : \iii \in Z(r) \}$ is a cover for $I^\infty$. We will
choose the constants $\as_\iii$ rigorously in a while. Next we study
how these separation conditions are related.

\begin{lemma} \label{thm:stable_a}
  Suppose a general IFS has bounded overlapping. Then it has also weak
  bounded overlapping.
\end{lemma}

\begin{proof}
  If the weak bounded overlapping is not satisfied, then the cardinality of
  incomparable subsets of $R(x) = \{ \iii \in I^* : x \in \fii_\iii(X)
  \}$ is not uniformly bounded as $x \in X$. Therefore, $\sup_{x \in
  X} \# \bigl( R(x) \cap Z(r) \bigr) \to \infty$, as $r \searrow
  0$. On the other hand, $R(x) \cap Z(r) \subset Z(x,r)$ for all $x
  \in X$ and $r > 0$, which gives a contradiction.
\end{proof}

It seems that by assuming only the mappings of a general IFS to be
Lipschitz it
is very difficult to get information about the Hausdorff dimension of
the limit set. While the Lipschitz condition provides us with an upper bound
for the diameter of the cylinder set, it does not give any kind of
lower bound for the size of the cylinder set. Having the lower
bound seems to be crucial for getting this kind of information.
Assuming the mappings $\fii_\iii$ to be bi--Lipschitz, we denote
the ``maximal derivative'' with
\begin{equation}
  L_\iii(x) = \limsup_{y \to x} \frac{|\fii_\iii(x) - \fii_\iii(y)|}{|x-y|}
\end{equation}
and the ``minimal derivative'' with
\begin{equation}
  l_\iii(x) = \liminf_{y \to x} \frac{|\fii_\iii(x) - \fii_\iii(y)|}{|x-y|}.
\end{equation}
We say that a
general IFS is \emph{bi--Lipschitz} if the mappings $\fii_\iii$ are
bi--Lipschitz and there exist cylinder functions $\apsi_\iii^t$ and
$\ypsi_\iii^t$ satisfying the chain rule such that $\apsi_\iii^t(\hhh) \le
l_\iii\bigl( \pi(\hhh) \bigr)^t$ and $\ypsi_\iii^t(\hhh) \ge
L_\iii\bigl( \pi(\hhh) \bigr)^t$ for all $\hhh \in I^\infty$, and in both
functions the parameter $t$ is an exponent, that is,
$\apsi_\iii^t(\hhh) = \bigl( \apsi_\iii^1(\hhh) \bigr)^t$ and
$\ypsi_\iii^t(\hhh) = \bigl( \ypsi_\iii^1(\hhh) \bigr)^t$. We also
assume that the bi--Lipschitz constants for the mappings $\fii_\iii$
are $\as_\iii = \inf_{\hhh \in I^\infty} \apsi_\iii^1(\hhh)$ and
$\ys_\iii = \sup_{\hhh \in I^\infty} \ypsi_\iii^1(\hhh)$. From now
on, these are the constants $\as_\iii$ we will use in the definition
of the bounded overlapping.

\begin{lemma} \label{thm:blifsssc}
  A bi--Lipschitz IFS satisfying the SSC has bounded overlapping.
\end{lemma}

\begin{proof}
  We use the idea found in the proof of Proposition 9.7 of Falconer
  \cite{fa3}.
  Denote $q = \min_{i \ne j} d\bigl( \fii_i(X),\fii_j(X)
  \bigr)$, where $d$ means the distance between two given
  sets, and take $x \in E$ and $r > 0$. We can take $x$ from $E$ since
  otherwise there is nothing to prove. Choose $\iii \in I^\infty$ such
  that $x = \pi(\iii)$. Since now $\fii_{\iii|_n}(X) \cap B(x,r) \ne
  \emptyset$ for every $n \in \N$, we can choose $n$ such that
  $\iii|_n \in Z(x,r)$. Take also an arbitrary $\jjj \in Z(r)$ such that
  $\jjj \ne \iii|_n$ and let $0 \le j < n$ be the largest integer for
  which $\jjj|_j = \iii|_j$. If it were
  $d\bigl( \fii_{\iii|_n}(X),\fii_\jjj(X) \bigr) <
  \as_{\iii|_j}q$, there would be $y \in \fii_{\iii|_n}(X)$ and $z
  \in \fii_\jjj(X)$ such that
  \begin{equation}
    |y-z| < \as_{\iii|_j} q.
  \end{equation}
  The bi--Lipschitz condition implies $|(\fii_{\iii|_j})^{-1}(y) -
  (\fii_{\iii|_j})^{-1}(z)| < q$, which contradicts the strong
  separation assumption due to the choice of $j$. Hence
  \begin{equation}
    d\bigl( \fii_{\iii|_n}(X),\fii_\jjj(X) \bigr) \ge
    \as_{\iii|_j} q \ge \as_{\iii|_{n-1}} q \ge rq
  \end{equation}
  and thus $\iii|_n$ is the only symbol in $Z(r)$ with
  $\fii_{\iii|_n}(X) \cap B(x,rq) \ne \emptyset$. This also means that
  there exists exactly one $\hhh \in Z(r/q)$ for which $\fii_\hhh(X)
  \cap B(x,r) \ne \emptyset$. Take now an arbitrary $\jjj \in 
  Z(x,r)$ and assuming $q<1$ we notice that $\jjj = \hhh,\kkk$ for
  some $\kkk \in
  I^*$. Choose the smallest integer $k$ such that $\as_\kkk <
  q/\underline{K}_1$ for all $\kkk \in I^*$ for which $|\kkk| \ge
  k$. Here $\underline{K}_t$ is the constant from the BVP of the
  cylinder function $\apsi_\iii^t$. Hence if it were $\jjj =
  \hhh,\kkk$ for some $\kkk \in I^*$ for which $|\kkk| > k$, it would
  hold that
  \begin{equation}
    \as_{\jjj|_{|\jjj|-1}} \le \underline{K}_1 \as_\hhh
    \as_{\kkk|_{|\kkk|-1}} < r
  \end{equation}
  and therefore $\jjj$ could not be in 
  $Z(x,r)$. Thus there can be at maximum $(\# I)^k$ of such $\kkk$
  and hence $\# Z(x,r) \le (\# I)^k$.
\end{proof}

It seems to be important that the shape of the open set of the
OSC would not be too ``wild'', and, therefore, the shape of the cylinder
sets, or rather the sets $\fii_\iii(X)$, is under
control. See also Theorem 4.9 of Graf, Mauldin and Williams
\cite{gmw}. Motivated by this, we say that the \emph{boundary
condition} is satisfied if there exists $\roo_0>0$ such that
\begin{equation}
  \inf_{x \in \partial X} \inf_{0<r<\roo_0} \frac{\HH^d\bigl( B(x,r) \cap
  \text{int}(X) \bigr)}{\HH^d\bigl( B(x,r) \bigr)} > 0,
\end{equation}
where $\partial X$ denotes the boundary of the set $X$. This condition
says that the boundary of $X$ cannot be too ``thick''; for example, recalling
the Lebesgue density theorem, we have $\HH^d(\partial X) = 0$. The boundary
condition is clearly satisfied if the set $X$ is convex.

\begin{proposition} \label{thm:blifsoscbc}
  A bi--Lipschitz general IFS satisfying the OSC and the boundary
  condition has weak bounded overlapping if $\ys_\iii / \as_\iii$ is
  bounded as $\iii \in I^*$.
\end{proposition}

\begin{proof}
  Fix $x \in X$ and denote with $R$ some incomparable subset of $\{
  \iii \in I^* : x \in \fii_\iii(X) \}$. Put $r_0 = \min\{
  \roo_0,d(X,\partial \Omega) \}$, where $\roo_0$ is as in the
  boundary condition. Now there exists
  $\delta > 0$ such that for every $y \in X$ we have
  \begin{equation}
    \HH^d\bigl( B(y,r) \cap \text{int}(X) \bigr) \ge \HH^d\bigl(
    B(y,\delta r) \bigr)
  \end{equation}
  whenever $0<r<r_0$. Note that the collection $\{ \fii_\iii\bigl(
  \text{int}(X) \bigr) : \iii \in R \}$ is disjoint due to the
  OSC. For each $\iii \in R$ take $y_\iii \in X$ such that
  $\fii_\iii(y_\iii) = x$ and choose an increasing sequence of finite
  sets $R_1 \subset R_2 \subset \cdots$ such that
  $\bigcup_{j=1}^\infty R_j = R$. Now fix $j$ and choose $r>0$ small
  enough such that $r_\iii := r/\as_\iii < r_0$ for all $\iii \in
  R_j$. Using now the boundary condition, bi--Lipschitzness and the
  OSC, we see that
  \begin{align}
    \# R_j r^d &= \sum_{\iii \in R_j} \as_\iii^d r_\iii^d \notag \\
    &= \bigl( \alpha(d)\delta^d \bigr)^{-1} \sum_{\iii \in R_j}
    \as_\iii^d \HH^d\bigl( B(y_\iii,\delta r_\iii) \bigr) \notag \\
    &\le \bigl( \alpha(d)\delta^d \bigr)^{-1} \sum_{\iii \in R_j}
    \as_\iii^d \HH^d\bigl( B(y_\iii,r_\iii) \cap \text{int}(X)
    \bigr) \notag \\
    &\le \bigl( \alpha(d)\delta^d \bigr)^{-1} \sum_{\iii \in R_j}
    \HH^d\Bigl( \fii_\iii\bigl( B(y_\iii,r_\iii) \cap
    \text{int}(X) \bigr) \Bigr) \\
    &\le \bigl( \alpha(d)\delta^d \bigr)^{-1} \HH^d \Biggl(
    \bigcup_{\iii \in R_j} B(x,\ys_\iii r_\iii) \Biggr) \notag \\
    &\le \delta^{-d} C^d r^d, \notag
  \end{align}
  where $\alpha(d)$ is the Hausdorff measure of the unit ball and
  $\ys_\iii / \as_\iii \le C$ as $\iii \in I^*$. Hence $\# R = \lim_{j
  \to \infty} \# R_j \le \delta^{-d} C^d$, where the upper bound does
  not depend on the choice of $x \in X$.
\end{proof}

Now we define an important class of iterated function systems.
We say that a general IFS is \emph{(weakly) geometrically stable} if it is
bi--Lipschitz and it has (weak) bounded overlapping.
Geometrically stable systems are clearly weakly geometrically stable
by Lemma \ref{thm:stable_a}. If we have a good control over the size
of the cylinder sets, the converse is also true.

\begin{proposition} \label{thm:ifssgs}
  Suppose a general IFS is weakly geometrically stable such that
  $\ys_\iii / \as_\iii$ is bounded as $\iii \in I^*$. Then it is also
  geometrically stable.
\end{proposition}

\begin{proof}
  Notice first that the weak bounded overlapping assumption implies the
  existence of the constant $C$ for which
  $\sum_{\iii \in A} \khii_{\fii_\iii(X)}(x) < C$ whenever $x
  \in X$ and the set $A \subset I^*$ is incomparable. Recall that
  \begin{equation}
    Z(x,r) = \{ \iii \in Z(r) : \fii_\iii(X) \cap B(x,r) \ne \emptyset \}
  \end{equation}
  is incomparable and notice that $\fii_\iii(X) \subset
  B(x,rd(X)\ys_\iii / \as_\iii + r)$ as $\iii \in Z(x,r)$. Choosing
  $C$ big enough such that also $d(X)\ys_\iii / \as_\iii + 1 \le C$ whenever
  $\iii \in I^*$, we get
  \begin{align}
    \# Z(x,r) r^d &\le \Bigl(\min_{i \in I}\as_i^d\Bigr)^{-1} \sum_{\iii \in
    Z(x,r)} \as_\iii^d \notag \\
    &\le \Bigl( \HH^d(X) \min_{i \in I}\as_i^d \Bigr)^{-1} \sum_{\iii
    \in Z(x,r)} \HH^d\bigl( \fii_\iii(X) \bigr) \\
    &\le \Bigl( \HH^d(X) \min_{i \in I}\as_i^d \Bigr)^{-1}
    \int_{B(x,Cr)} \sum_{\iii \in Z(x,r)} \khii_{\fii_\iii(X)}(x)
    d\HH^d(x). \notag
  \end{align}
  Since $r^d = \bigl( \alpha(d) C^d \bigr)^{-1} \HH^d\bigl( B(x,Cr)
  \bigr)$, we conclude
  \begin{equation} \label{eq:boundeddi}
    \# Z(x,r) \le \frac{\alpha(d)C^{d+1}}{\HH^d(X)\min_{i \in I} \as_i^d},
  \end{equation}
  where $\alpha(d)$ is the Hausdorff measure of the unit ball.
\end{proof}

Before studying the Hausdorff dimension of the limit set, we show in
the following theorem that with respect to any invariant measure we
can have the same structure in the limit set as in the symbol space.
Under the weak bounded overlapping assumption, somehow the weakest
separation condition, we can project any invariant measure from
$I^\infty$ to the limit set $E$ such that the overlapping has measure
zero.

\begin{theorem} \label{thm:leikkausmittanolla}
  Suppose a general IFS has weak bounded overlapping. Then for $m = \mu
  \pallo \pi^{-1}$, where $\mu \in \MM_\sigma(I^\infty)$, we have
  \begin{equation}
    m\bigl( \fii_\iii(X) \cap \fii_\jjj(X) \bigr) = 0
  \end{equation}
  whenever $\iii$ and $\jjj$ are incomparable.
\end{theorem}

\begin{proof}
  We use the idea found in the proof of Lemma 3.10 of Mauldin and
  Urba\'nski \cite{mu}.
  For fixed incomparable $\hhh$ and $\kkk$ we denote $A = \fii_\hhh(X)
  \cap \fii_\kkk(X)$ and
  $A_n = \bigcup_{\iii \in I^n} \fii_\iii(A)$ as $n \in \N$. Let us
  first show that $\bigcap_{q=1}^\infty \bigcup_{n=q}^\infty A_n =
  \emptyset$. Assume contrarily that there exists $x \in
  \bigcap_{q=1}^\infty \bigcup_{n=q}^\infty A_n$. Then $x \in
  \bigcup_{n=q}^\infty A_n$ for every $q$ and hence $x \in A_{n_q}$,
  where $\{ n_q \}_{q \in \N}$ is an increasing sequence of indexes. Now
  for each $q$ there exists a symbol $\jjj_q \in I^{n_q}$ such that $x
  \in \fii_{\jjj_q,\hhh}(X)$ and $x \in \fii_{\jjj_q,\kkk}(X)$. Denoting with
  $R_k^*$ the maximal incomparable subset of $R_k = \{ \iii \in
  \bigcup_{q=1}^k (I^{n_q+|\hhh|} \cup I^{n_q+|\kkk|}) : x \in
  \fii_\iii(X) \}$, we have $\# R_1 \ge
  2$ and also $\# R_1^* \ge 2$. Clearly, $\# R_2 \ge 4$, and even if it were
  $\jjj_2|_{n_1+|\hhh|} = \jjj_1,\hhh$ (or $\jjj_2|_{n_1+|\kkk|} =
  \jjj_1,\kkk$), it is still $\# R_2^* \ge 3$ since the two new
  symbols $\jjj_2,\hhh$ and $\jjj_2,\kkk$ with the symbol
  $\jjj_1,\kkk$ (or $\jjj_1,\hhh$) are incomparable. Observe that for
  each $k$ the symbol $\jjj_k$ can be comparable at maximum with one
  element of $R^*_{k-1}$. Thus continuing in this manner, we get $\# R_k^*
  \ge k+1$ as $k \in \N$. The claim is proved since this contradicts
  the bounded overlapping assumption.

  The boundedness assumption also implies $\sum_{\iii \in I^n}
  \khii_{\fii_\iii(A)}(x) \le C$ for every $x \in X$ and $n \in \N$
  with some constant $C \ge 0$. Thus, using the invariance of $\mu$, we
  have
  \begin{align}
    m(A_n) &= m\Biggl( \bigcup_{\iii \in I^n} \fii_\iii(A) \Biggr) \ge
    C^{-1}\sum_{\iii \in I^n} m\bigl( \fii_\iii(A) \bigr) \notag \\
    &\ge C^{-1}\sum_{\iii \in I^n} \mu\bigl( [\iii;\pi^{-1}(A)] \bigr)
    = C^{-1}\mu \pallo \sigma^{-n}\bigl( \pi^{-1}(A) \bigr) = C^{-1}m(A)
  \end{align}
  whenever $n \in \N$. So, if $m(A) > 0$, we get a contradiction
  immediately since
  \begin{equation}
    m\Biggl( \bigcap_{q=1}^\infty \bigcup_{n=q}^\infty A_n \Biggr) =
    \lim_{q \to \infty} m\Biggl( \bigcup_{n=q}^\infty A_n \Biggr) \ge
    \lim_{q \to \infty} m(A_q) \ge C^{-1}m(A).
  \end{equation}
  The proof is complete.
\end{proof}

If we assume that the cylinder function satisfy
\begin{equation} \label{eq:bicylinder}
  \apsi_\iii^t(\hhh) \le \psi_\iii^t(\hhh) \le \ypsi_\iii^t(\hhh),
\end{equation}
where $\iii \in I^*$, then we clearly have
$\dimeala(I^\infty) \le \dime(I^\infty) \le \dimeyla(I^\infty)$, where
$\dimeala(I^\infty)$ and $\dimeyla(I^\infty)$ are the equilibrium dimensions
derived from cylinder functions $\apsi_\iii^t$ and $\ypsi_\iii^t$,
respectively. The following theorem guarantees that the similar behaviour
occurs also for the Hausdorff dimension with geometrically stable
systems. It is now very tempting to guess that in some cases making a
reasonable choice for the cylinder function, it is possible to get $\dimh(E) =
\dime(I^\infty)$. If there is no danger of misunderstanding, we call
also the projected equilibrium measure an equilibrium measure. 

\begin{theorem} \label{thm:stabledim}
  Suppose a general IFS is geometrically stable. Then it has
  \begin{equation} \label{eq:dimearvio}
    \dimeala(I^\infty) \le \dimh(E) \le \dimeyla(I^\infty),
  \end{equation}
  and, in fact, $\HH^t(A) > 0$ as $t \le \dimeala(I^\infty)$ whenever $A$
  is a Borel set such that $\underline{m}(A) = 1$ and
  $\underline{m}$ is the equilibrium measure constructed using the
  cylinder function $\apsi_\iii^t$.
\end{theorem}

\begin{proof}
  Let us first prove the right--hand side of (\ref{eq:dimearvio}). For
  each $t \ge 0$ we have
  \begin{align}
    \HH^t(E) &\le \lim_{n \to \infty}\inf \biggl\{ \sum_j d\bigl(
    \fii_{\iii_j}(E) \bigr)^t : E \subset \bigcup_j \fii_{\iii_j}(E),\;
    |\iii_j| \ge n \biggr\} \notag \\
    &\le \lim_{n \to \infty}\inf \biggl\{ \sum_j d(E)^t\ys_{\iii_j}^t :
    E \subset \bigcup_j \fii_{\iii_j}(E),\; |\iii_j| \ge n \biggr\} \\
    &\le \overline{K}_t d(E)^t \overline{\GG}^t(I^\infty), \notag
  \end{align}
  where $\overline{\GG}^t$ is the measure constructed in a similar way
  as in (\ref{eq:falconer_measure1}) and (\ref{eq:falconer_measure2})
  but using the cylinder function $\ypsi_\iii^t$. Here
  $\overline{K}_t$ is the constant of the BVP. Thus $\dimh(E) \le
  \dimeyla(I^\infty)$. Notice that here we did not need any kind of
  separation condition.

  For the left--hand side recall first that the set $Z(r)$
  is incomparable and the cardinality of the set $Z(x,r)$
  is bounded as $x \in X$ and $0<r<r_0=r_0(x)$. Now for fixed $x \in
  X$ and $0<r<r_0(x)$ we have, using 
  theorems \ref{thm:aakolme}, \ref{thm:equivalent} and \ref{thm:eqdim},
  \begin{align}
    \underline{m}\bigl( B(x,r) \bigr) &\le \sum_{\iii \in Z(x,r)}
    \underline{m}\bigl( \fii_\iii(X) \bigr) \notag \\
    &\le \underline{K}_t \sum_{\iii \in Z(x,r)} \int_{I^\infty}
    \apsi_\iii^t(\hhh) d\underline{\nu}(\hhh) \le \underline{K}_t^2 \#
    Z(x,r) r^t,
  \end{align}
  where $\underline{m}$ and $\underline{\nu}$ are the corresponding
  equilibrium measure and conformal measure constructed using the
  cylinder function $\apsi_\iii^t$ and $t =
  \dimeala(I^\infty)$. Taking $A \subset E$ such that
  $\underline{m}(A) = 1$ and defining $A_k = \{ x \in A : \tfrac{1}{k}
  < r_0(x) \}$, we have $A = \bigcup_{k=1}^\infty A_k$. Now for each $x
  \in A_k$ we have
  \begin{equation}
    \frac{\underline{m}\bigl( B(x,r) \bigr)}{r^t} \le \underline{K}_t^2
    \# Z(x,r)
  \end{equation}
  as $0<r<\tfrac{1}{k}$, and thus $\HH^t(A_k) \ge
  C\underline{m}(A_k)$ for some positive constant $C$. Since $\HH^t(A)
  = \lim_{k \to \infty} \HH^t(A_k) \ge C > 0$, we have finished the
  proof.
\end{proof}

Next we introduce a couple of examples of IFS's which have aroused
great interest for some time. After each definition we also discuss a
little how our theory turns out to be in that particular case. Our
main application is the self--affine case described below.

\begin{definition}
  Let the mappings of IFS be \emph{similitudes}, that is, for each $i \in
  I$ there exists $0<s_i<1$ such that $|\fii_i(x)-\fii_i(y)| =
  s_i|x-y|$ whenever $x,y \in \Omega$. We call this kind of setting a
  \emph{similitude IFS} and the corresponding limit set a
  \emph{self--similar set}.
\end{definition}

If for each $\iii \in I^*$ we choose
$\psi_\iii^t \equiv s_\iii^t$, where $s_\iii = s_{i_1}\cdots
s_{i_{|\iii|}}$, then $\psi_\iii^t$ is a constant cylinder function
satisfying the chain rule. Assuming weak bounded overlapping, the
similitude IFS is geometrically stable due to Proposition
\ref{thm:ifssgs}, and, thus, with this choice of the cylinder function
we get, applying Theorem \ref{thm:stabledim}, that $\dimh(E) =
\dime(I^\infty)$ (we clearly have $\as_\iii = \ys_\iii = s_\iii$).
Notice also that Theorem \ref{thm:blifsoscbc} provides us with
concrete assumptions, namely the OSC and the boundary condition, to obtain the
weak bounded overlapping. The definition of
this setting goes back to the well known article of Hutchinson
\cite{hu}. However, the open set condition was first introduced by
Moran in \cite{mo}. Schief studied in \cite{sc}, extending ideas of
Bandt and Graf \cite{bg}, the relationship between the OSC and the
choice of the mappings of IFS. It also follows from the result of
Schief that the weak bounded overlapping implies the OSC since
according to Proposition \ref{thm:ifssgs} and Theorem
\ref{thm:stabledim} we have $\HH^t(E) > 0$, where $t=\dimh(E)$.
For example, using Theorems \ref{thm:aakolme},
\ref{thm:equivalent} and \ref{thm:leikkausmittanolla}, we see that
the $t$--equilibrium measure, where $t = \dimh(E)$, gives us the idea
of ``mass distribution''; we start with mass $1$ and on each level of
the construction we divide the mass from cylinder sets of the
previous level using the rule obtained by the probability vector
$(s_i^t)_{i \in I}$.

\begin{definition}
  Suppose $d \ge 2$. Let mappings of IFS be $C^1$ and \emph{conformal}
  on an open set $\Omega_0 \supset \overline{\Omega}$. Hence $|\fii_i'|^d =
  |J_{\fii_i}|$ for every $i \in I$, where $J$ stands for the usual
  Jacobian and the norm on the left--hand side is just a standard
  ``sup--norm'' for linear mappings. We call this kind of setting a
  \emph{conformal IFS} and the corresponding limit set a
  \emph{self--conformal set}. 
\end{definition}

Observe that the conformal mapping is
complex analytic in the plane and, by Liouville's theorem, a M\"obius
transformation in higher dimensions (see Theorem 4.1 of Reshetnyak
\cite{re}). So, in fact, conformal mappings are $C^\infty$ and
infinitesimally similitudes. Notice also that it is essential to use
the bounded set $\Omega$ here since conformal mappings contractive
in the whole $\R^d$ are similitudes. If for each
$\iii \in I^*$ we choose $\psi_\iii^t(\hhh) = \bigl| \fii_\iii'\bigl(
\pi(\hhh) \bigr) \bigr|^t$, then $\psi_\iii^t$ is a cylinder function
satisfying the chain rule. The BVP for $\psi_\iii^t$ is guaranteed by
the smoothness of mappings $\fii_\iii$, Proposition
\ref{thm:kakspisteyks} and the chain rule. With this choice of the cylinder
function we may also call the BVP a \emph{bounded distortion property
(BDP)} since it gives information about the distortion of mappings
$\fii_\iii$. Assuming weak bounded overlapping, the system is
geometrically stable and we get $\dimh(E) = \dime(I^\infty)$ like
before (we can choose $\apsi_\iii^t = \ypsi_\iii^t = \psi_\iii^t$).
Notice again that, using Theorem \ref{thm:blifsoscbc}, the OSC
and the bounded overlapping provides us with a sufficient condition for the
weak bounded overlapping to hold.
In the conformal case the equilibrium measure is equivalent to the conformal
measure. Peres, Rams, Simon and Solomyak \cite{pr} generalised the
result of Schief for the conformal setting. Thus, the weak bounded
overlapping implies the OSC also in this setting. Mauldin and
Urba\'nski \cite{mu} have introduced the theory of conformal IFS's for
infinite collections of mappings.

\begin{definition}
  Let the mappings of IFS be \emph{affine}, that is, $\fii_i(x) = A_ix+a_i$ for
  every $i \in I$, where $A_i$ is a contractive non--singular linear
  mapping and $a_i \in \R^d$.
  We call this kind of setting an \emph{affine IFS} and the
  corresponding limit set a \emph{self--affine set}.
\end{definition}

Clearly, the products $A_\iii =
A_{i_1}\cdots A_{i_{|\iii|}}$ are also contractive and
non--singular. Singular values of a non--singular matrix are the
lengths of the principle semiaxes of the image of the unit ball. On
the other hand, the singular values $1 > \alpha_1 \ge \alpha_2 \ge
\cdots \ge \alpha_d > 0$ of a contractive, non--singular matrix $A$
are the non--negative square roots of the eigenvalues of $A^*A$,
where $A^*$ is the transpose of $A$. Define the singular value function
$\alpha^t$ by setting $\alpha^t(A) =
\alpha_1\alpha_2\cdots\alpha_{l-1}\alpha_l^{t-l+1}$, where $l$ is
the smallest integer greater than or equal to $t$. For all $t>d$ we put
$\alpha^t(A) = (\alpha_1\cdots\alpha_d)^{t/d}$. It is clear that
$\alpha^t(A)$ is continuous and strictly decreasing in $t$. If for each
$\iii \in I^*$ we choose $\psi_\iii^t \equiv \alpha^t(A_\iii)$, then
$\psi_\iii^t$ is a constant cylinder function. The subchain rule for
$\psi_\iii^t$ is satisfied by Lemma 2.1 of Falconer \cite{fa1}.
Since in this case we do not have the chain rule, it is still very difficult
to say anything ``concrete'' about the equilibrium measure or the
Hausdorff dimension of the limit set.
Assuming the SSC, we have bounded overlapping satisfied by
Lemma \ref{thm:blifsssc} and thus we can at least approximate the
Hausdorff dimension of the limit set by using Theorem \ref{thm:stabledim}.
We study self--affine sets and equilibrium measures of affine
IFS's in more detail in the next chapter.
The following example shows us that in the affine setting we cannot
allow overlapping even at one single point if we want to have the
weak bounded overlapping.

\begin{example}
  Put $I = \{ 1,2 \}$, $X = \overline{B(0,1)} \cap \{ (x_1,x_2) \in
  \R^2 : |x_2| \le x_1 \}$ and define two affine mappings (in matrix
  notation) as follows:
  \begin{align}
    \fii_1(x_1,x_2) &= \begin{pmatrix}
                         \cos(\pi/8) & -\sin(\pi/8) \\
                         \sin(\pi/8) & \cos(\pi/8)
                       \end{pmatrix}\begin{pmatrix}
                         0.9 & 0 \\
                         0   & 0.3
                       \end{pmatrix}\begin{pmatrix}
                         x_1 \\ x_2
                       \end{pmatrix} \notag \\
   \fii_2(x_1,x_2) &=  \begin{pmatrix}
                         \cos(\pi/8) & \sin(\pi/8) \\
                         -\sin(\pi/8) & \cos(\pi/8)
                       \end{pmatrix}\begin{pmatrix}
                         0.9 & 0 \\
                         0   & 0.3
                       \end{pmatrix}\begin{pmatrix}
                         x_1 \\ x_2
                       \end{pmatrix}.
  \end{align}
  The set $X$ is a sector with angle $\pi/2$, and functions $\fii_1$ and
  $\fii_2$ map this sector into two flattened sectors inside $X$ such
  that $\fii_1(X) \cap \fii_2(X) = \{ 0 \}$. The OSC is
  therefore satisfied. Since the origin is the only fixed point of
  both mappings, the limit set is nothing but $\{ 0 \}$. This setting
  does not satisfy the weak bounded overlapping, because the amount of
  cylinder sets of the level $n$ including the origin is always $2^n$.
\end{example}

Notice that the similitude IFS is always both conformal and affine. Also if we
consider the cylinder functions introduced before, we notice that the
cylinder function of the similitude IFS is just a special case of both
cylinder functions of conformal IFS and affine IFS.
We could also study more general limit sets in this manner. Falconer
\cite{fa2} has obtained some dimension results into this direction by
using the singular value function for the derivatives of more general
mappings. Using the concept of general IFS, it is possible to use
bi--Lipschitz mappings for defining geometric constructions for which
it is possible easily to determine the Hausdorff dimension of the
limit set.

\begin{example}
  Consider a bi--Lipschitz general IFS satisfying the OSC and the
  boundary condition. Suppose that for each $\iii \in I^*$ there exist balls
  $\underline{B}_\iii$ and $\overline{B}_\iii$ and a constant $C>0$ such that
  \begin{equation} \label{eq:pallot}
    \underline{B}_\iii \subset \fii_\iii(X) \subset \overline{B}_\iii,
  \end{equation}
  $l_\iii(x) \ge Cd(\underline{B}_\iii)$ and $L_\iii(x) \le
  Cd(\overline{B}_\iii)$ as $x \in X$.
  Now, if the ratio between the radii of $\overline{B}_\iii$ and
  $\underline{B}_\iii$ remains bounded, then
  \begin{equation}
    \dimh(E) = t,
  \end{equation}
  where $t \ge 0$ is the unique number satisfying
  \begin{equation}
    \lim_{n \to \infty} \tfrac{1}{n} \log \sum_{\iii \in I^n} r_\iii^t
    = 0
  \end{equation}
  and $r_\iii$ is the radius of either $\overline{B}_\iii$ or
  $\underline{B}_\iii$. This result is easily obtained by first noting
  that the ratio $\ys_\iii / \as_\iii$ is bounded
  and then using Propositions \ref{thm:blifsoscbc}
  and \ref{thm:ifssgs}, Theorem \ref{thm:stabledim} and recalling the
  definition of the topological pressure.
\end{example}

The concept of the general IFS is also crucial in the following example,
which says that the relative positions of cylinder sets are irrelevant
concerning the Hausdorff dimension of the limit set of conformal
systems provided that a sufficient separation condition is satisfied.

\begin{example}
  Consider a conformal IFS satisfying the OSC and the boundary
  condition. Choosing $\psi_\iii^t(\hhh) = \bigl| \fii_\iii'\bigl(
  \pi(\hhh) \bigr) \bigr|^t$, we have $\dimh(E) = \dime(I^\infty)$. In
  this setting the placement of cylinder sets is fixed and their
  relative positions follow from the rule obtained by the mappings
  $\fii_i$. We could now rearrange the placements and ask what happens
  to the Hausdorff dimension of the limit set. We
  define a general IFS by composing our original conformal mappings
  with isometries such that the OSC remains satisfied. Since this
  does not affect our cylinder function and composed mappings are
  still conformal, we will get for the limit set $\tilde{E}$ of this
  general IFS that $\dimh(\tilde{E}) = \dime(I^\infty)$ using
  Propositions \ref{thm:blifsoscbc}, \ref{thm:ifssgs} and Theorem
  \ref{thm:stabledim}.
\end{example}

\section{Dimension of the equilibrium measure}

We say that the Hausdorff dimension of a given Borel probability measure
$m$ is $\dimh(m) = \inf\{ \dimh(A) : A \text{ is a Borel set such that }
m(A)=1 \}$. To check if $\dimh(m)$ $= \dimh(E)$ is one way to examine
how well a given measure $m$ is spread out on a given set $E$. If we
consider similitude and conformal IFS's and we choose cylinder
functions to be the ones introduced in the previous chapter, we notice
using Proposition \ref{thm:ifssgs} and Theorem \ref{thm:stabledim}
that $\dimh(m) = \dimh(E) =: t$ provided that the weak bounded
overlapping is satisfied. Here $m$ and $E$ are the
corresponding $t$--equilibrium measure and the limit set. It is an
interesting question whether we can obtain the same result for the affine
setting. In the following we will prove that at least in ``almost all'' affine
cases this is possible. To do that we first have to prove that the
equilibrium measure $\mu$ is \emph{ergodic}, that is, $\mu(A)=0$ or
$\mu(A)=1$ for every Borel set $A$ for which $A = \sigma^{-1}(A)$.
In the proof we use some ideas found in Zinsmeister \cite{zi}, Bowen
\cite{bo} and Phelps \cite{ph}.

\begin{theorem} \label{thm:ergodisuus}
  There exists an ergodic equilibrium measure.
\end{theorem}

\begin{proof}
  Let us first study mappings $\PP,\QQ_n,\QQ :
  \MM_\sigma(I^\infty) \to \R$, for which $\PP(\mu) = h_\mu$,
  $\QQ_n(\mu) = \tfrac{1}{n} \sum_{\iii \in I^n} \mu([\iii])
  \log\psi_\iii^t(\hhh)$ and $\QQ(\mu) = \lim_{n \to \infty}
  \QQ_n(\mu) = E_\mu(t)$. It is clear that each $\QQ_n$ is affine and
  continuous (basically because cylinder sets have empty boundary) and
  $\QQ$ is affine. We will prove that $\PP$ is affine and upper
  semicontinuous.

  Fix $0 \le x_1,x_2 \le 1$ and $\lambda \in [0,1]$ and denote $x =
  \lambda x_1 + (1-\lambda) x_2$. Now using the concavity of the function
  $H(x) = -x \log x$, $H(0)=0$, we have
  \begin{align}
    0 &\le -x\log x + \lambda x_1\log x_1 + (1-\lambda)x_2\log x_2 \notag \\
    &= -\lambda x_1(\log x - \log x_1) - (1-\lambda)x_2(\log x - \log
       x_2) \notag \\ 
    &= -\lambda x_1\bigl( \log x - \log(\lambda x_1) \bigr) -
       (1-\lambda)x_2\bigl( \log x - \log( (1-\lambda)x_2 )
       \bigr) \notag \\
    &\quad\, -\lambda x_1\log\lambda - (1-\lambda)x_2\log(1-\lambda) \\
    &\le -x_1\lambda\log\lambda - x_2(1-\lambda)\log(1-\lambda) \notag \\
    &\le x_1\tfrac{1}{e} + x_2\tfrac{1}{e} \notag
  \end{align}
  since $\log x - \log(\lambda x_1)$ and $\log x -
  \log((1-\lambda)x_2)$ are positive. Hence we get
  \begin{align} \label{eq:arvio}
    0 &\le \sum_{\iii \in I^n} H\bigl( \mu([\iii]) \bigr) - \lambda
    \sum_{\iii \in I^n} H\bigl( \mu_1([\iii]) \bigr) - (1-\lambda)
    \sum_{\iii \in I^n} H\bigl( \mu_2([\iii]) \bigr) \notag \\
    &\le \tfrac{1}{e} \sum_{\iii \in I^n} \mu_1([\iii]) + \tfrac{1}{e}
    \sum_{\iii \in I^n} \mu_2([\iii]) = \tfrac{2}{e},
  \end{align}
  where $\mu_1,\mu_2 \in \MM_\sigma(I^\infty)$ and $\mu = \lambda\mu_1
  + (1-\lambda)\mu_2$. By the convexity of $\MM_\sigma(I^\infty)$ we
  have $\mu \in \MM_\sigma(I^\infty)$ and thus it follows from
  (\ref{eq:arvio}) that
  $h_\mu = \lambda h_{\mu_1} + (1-\lambda) h_{\mu_2}$, and hence,
  $\PP$ is affine. Take next $\eps > 0$ and $\mu \in
  \MM_\sigma(I^\infty)$ and choose $n_0$ big enough such that
  \begin{equation}
    \tfrac{1}{n} \sum_{\iii \in I^n} H\bigl( \mu([\iii]) \bigr) \le
    h_\mu + \tfrac{\eps}{2}
  \end{equation}
  whenever $n \ge n_0$. Now we choose arbitrary $\eta \in
  \MM_\sigma(I^\infty)$ for which
  \begin{equation}
    \tfrac{1}{n} \sum_{\iii \in I^n} H\bigl( \eta([\iii]) \bigr) \le
    \tfrac{1}{n} \sum_{\iii \in I^n} H\bigl( \mu([\iii]) \bigr) +
    \tfrac{\eps}{2}
  \end{equation}
  for some $n \ge n_0$. This choice can be made just by taking $\eta$
  to be close enough to $\mu$ in the weak topology and recalling that
  cylinder sets have empty boundary. Therefore, using Proposition
  \ref{thm:perusomin}(3), we have
  \begin{equation}
    h_\eta \le \tfrac{1}{n} \sum_{\iii \in I^n} H\bigl( \eta([\iii])
    \bigr) \le \tfrac{1}{n} \sum_{\iii \in I^n} H\bigl( \mu([\iii])
    \bigr) + \tfrac{\eps}{2} \le h_\mu + \eps
  \end{equation}
  for some $n \ge n_0$. We have
  established the upper semicontinuity of the mapping $\PP$.

  Denote the set of all ergodic measures of $\MM_\sigma(I^\infty)$
  with $\EE_\sigma(I^\infty)$. Let us now assume contrarily that $\PP +
  \QQ$ cannot attain its supremum with an ergodic measure, that is,
  $(\PP+\QQ)(\eta) < (\PP+\QQ)(\mu)$ for all $\eta \in
  \EE_\sigma(I^\infty)$, where $\mu$ is an equilibrium measure.
  Recalling Theorem 6.10 of Walters \cite{wa}, we know that the set
  $\MM_\sigma(I^\infty)$ is compact and convex and the set of its extreme
  points is exactly the set $\EE_\sigma(I^\infty)$. An extreme point of a
  convex set is a point which cannot be expressed as an average of two
  distinct points. Using Choquet's theorem (see Chapter 3 of
  \cite{ph}), we can get an ergodic
  decomposition for every invariant measure, namely, for each $\mu \in
  \MM_\sigma(I^\infty)$ there exists a Borel regular
  probability measure $\tau_\mu$ on $\EE_\sigma(I^\infty)$ such that
  \begin{equation} \label{eq:erg1}
    \RR(\mu) = \int_{\EE_\sigma(I^\infty)} \RR(\eta) d\tau_\mu(\eta)
  \end{equation}
  for every continuous affine $\RR : \MM_\sigma(I^\infty) \to \R$.

  Denoting now $A_k = \{ \eta \in \EE_\sigma(I^\infty) : (\PP+\QQ)(\mu)
  - (\PP+\QQ)(\eta) \ge \tfrac{1}{k} \}$, where $\mu$ is an
  equilibrium measure, we have $\bigcup_{k=1}^\infty A_k =
  \EE_\sigma(I^\infty)$ and thus $\tau_\mu(A_k) > 0$ for some $k$. Clearly,
  \begin{align}
    (\PP+\QQ)(\mu) - \int_{\EE_\sigma(I^\infty)} &(\PP+\QQ)(\eta)
    d\tau_\mu(\eta)
    = \int_{\EE_\sigma(I^\infty)} (\PP+\QQ)(\mu) - (\PP+\QQ)(\eta)
    d\tau_\mu(\eta) \notag \\
    &\ge \int_{A_k} \tfrac{1}{k} d\tau_\mu(\eta) = \tfrac{1}{k}
    \tau_\mu(A_k)
  \end{align}
  for every $k$ and thus
  \begin{equation} \label{eq:erg0}
    (\PP+\QQ)(\mu) > \int_{\EE_\sigma(I^\infty)} (\PP+\QQ)(\eta)
    d\tau_\mu(\eta).
  \end{equation}
  We will show that this is impossible, and, hence, the contradiction we
  obtain finishes the proof.

  Our goal now is to prove that we can write (\ref{eq:erg1}) also by
  using upper semicontinuous affine functions, particularly with $\PP +
  \QQ_n$. Fix $n \in \N$ and define $\overline{\RR} :
  \MM_\sigma(I^\infty) \to \R$ by setting $\overline{\RR}(\mu) =
  \inf\{ \RR(\mu) : \RR \ge \PP+\QQ_n \text{ is continuous and affine}
  \}$. Let us first prove that for each continuous affine $\RR_1,\RR_2
  > \PP+\QQ_n$ there exists a continuous affine $\RR$ for which
  $\PP+\QQ_n < \RR \le \RR_1,\RR_2$. Since $\PP+\QQ_n$ is affine and
  upper semicontinuous, we notice that the set $D = \{ (\mu,r) : \mu
  \in \MM_\sigma(I^\infty),\; r \le (\PP+\QQ_n)(\mu) \}$ is closed and
  convex. Since both mappings $\RR_i$ are continuous and affine as
  $i=1,2$, we get that both sets $D_i = \{ (\mu,r) : \mu \in
  \MM_\sigma(I^\infty),\; r=\RR_i(\mu) \}$ are compact and
  convex. Observe that the convex hull of the union $D_1 \cup D_2$ is
  compact and disjoint from the set $D$. Now applying the separation
  theorem for convex sets (Corollary 1.2 of \cite{ek}), we notice there
  exists a non--zero continuous real--valued linear functional $l$ on
  $\MM_\sigma(I^\infty) \times \R$ and a real number $\alpha$ such
  that the affine hyperplane
  \begin{equation}
    A = \{ (\mu,r) : \mu \in \MM_\sigma(I^\infty),\; l(\mu,r)=\alpha \}
  \end{equation}
  strictly separates the sets $D$ and the convex hull of $D_1 \cup
  D_2$. Because of the linearity of $l$, for each $\mu \in
  \MM_\sigma(I^\infty)$ there exists exactly one $r$ for which
  $(\mu,r) \in A$. Thus there exists a function $\RR :
  \MM_\sigma(I^\infty) \to \R$ such that $l\bigl( \mu,\RR(\mu) \bigr)
  = \alpha$ as $\mu \in \MM_\sigma(I^\infty)$. The function $\RR$ is
  affine and continuous because the functional $l$
  is linear and continuous. Since now $l(\mu,r) > \alpha$ for every
  $(\mu,r) \in D$ and $l(\mu,r) < \alpha$ for every $(\mu,r)$ in the
  convex hull of $D_1 \cup D_2$ (or the other way around), we have
  $\RR(\mu) > (\PP+\QQ_n)(\mu)$ and $\RR(\mu) < \RR_1(\mu),\RR_2(\mu)$
  for each $\mu \in \MM_\sigma(I^\infty)$, which is exactly what we
  wanted. A similar reasoning implies that
  $\overline{\RR} = \PP+\QQ_n$. Assume contrarily that there exists
  $\nu$ such that $(\PP+\QQ_n)(\nu) < \overline{\RR}(\nu)$. Now the
  set $D$ is disjoint from the compact convex set $\bigl\{
  \bigl(\nu,\overline{\RR}(\nu)\bigr) \bigr\}$ and the separation
  theorem gives us an immediate
  contradiction. We will next show that
  \begin{align} \label{eq:erg2}
    \int_{\EE_\sigma(I^\infty)} (\PP+\QQ_n)(\eta) d\tau_\mu(\eta) =
    \inf\biggl\{ \int_{\EE_\sigma(I^\infty)} \RR(\eta) &d\tau_\mu(\eta)
    : \RR \ge \PP+\QQ_n \notag \\ &\text{is continuous and
    affine} \biggr\}.
  \end{align}
  Let us denote with $\gamma$ the right--hand side of (\ref{eq:erg2})
  and choose a sequence $\{ \RR_i \}_{i \in \N}$ of continuous affine
  mappings greater than or equal to $\PP+\QQ_n$ such that
  \begin{equation}
    \lim_{i \to \infty} \int_{\EE_\sigma(I^\infty)} \RR_i(\eta)
    d\tau_\mu(\eta) = \gamma.
  \end{equation}
  We can assume that this sequence is monotonically decreasing, and
  hence there exists a Borel measurable function $\RR = \lim_{i \to
  \infty} \RR_i$ with $\RR \ge \PP+\QQ_n$ and
  \begin{equation}
    \int_{\EE_\sigma(I^\infty)} \RR(\eta) d\tau_\mu(\eta) = \gamma
  \end{equation}
  using the monotone convergence theorem. If it held that
  $\tau_\mu\bigl(\{ \eta \in \EE_\sigma(I^\infty) : \RR(\eta) >
  (\PP+\QQ_n)(\eta) \}\bigr) > 0$, then there would be real numbers $r$ and $q$
  such that also the set $\{ \eta \in \EE_\sigma(I^\infty) :
  (\PP+\QQ_n)(\eta) < r < q < \RR(\eta) \}$ has positive measure. By
  the Borel regularity, this set contains a compact subset $C$ of
  positive measure. Now for each $\eta \in C$ there is a continuous
  affine mapping $\tilde{\RR} \ge \PP+\QQ_n$ such that
  $\tilde{\RR}(\eta) < r$. Relying now on compactness and continuity,
  we can choose a finite number of them, say,
  $\tilde{\RR}_1,\ldots,\tilde{\RR}_k$ such that for each $\eta \in C$
  there is $1 \le j \le k$ with $\tilde{\RR}_j(\eta) < r$. For each $i
  \in \N$ we choose a continuous affine mapping $\hat{\RR}_i$ such
  that $\PP+\QQ_n < \hat{\RR}_i \le
  \RR_i,\tilde{\RR}_1,\ldots,\tilde{\RR}_k$. Hence $\hat{\RR}_i < r <
  r + \RR - q < \RR_i - (q-r)$ on $C$ and $\hat{\RR}_i \le \RR_i$
  elsewhere. Therefore,
  \begin{equation}
    \gamma \le \int_{\EE_\sigma(I^\infty)} \hat{\RR}_i(\eta)
    d\tau_\mu(\eta) \le \int_{\EE_\sigma(I^\infty)} \RR_i(\eta)
    d\tau_\mu(\eta) - (q-r)\tau_\mu(C),
  \end{equation}
  which finishes the proof of (\ref{eq:erg2}) as we let $i \to \infty$.
  Using now (\ref{eq:erg2}) and (\ref{eq:erg1}), we get that
  \begin{align}
    \int_{\EE_\sigma(I^\infty)} &(\PP+\QQ_n)(\eta) d\tau_\mu(\eta) =
    \inf\biggl\{ \int_{\EE_\sigma(I^\infty)} \RR(\eta)
    d\tau_\mu(\eta) : \RR \ge \PP+\QQ_n \notag \\
    &\qquad\qquad\qquad\qquad\quad\qquad\qquad\qquad\qquad\text{is
    continuous and affine} \biggr\} \\
    &= \inf\bigl\{ \RR(\mu) : \RR \ge \PP+\QQ_n \text{ is continuous
    and affine} \bigr\} = (\PP+\QQ_n)(\mu). \notag
  \end{align}
  Letting $n \to \infty$ and using the dominated convergence theorem,
  we have shown that (\ref{eq:erg0}) cannot happen and thus finished
  the proof.
\end{proof}

The ergodicity of the equilibrium measure is crucial in the following
proposition, which, for example, in the similitude and conformal cases
gives information about the so called local Hausdorff dimension of the
equilibrium measure. Compare it to Proposition 10.4 of Falconer
\cite{fa4}.

\begin{proposition} \label{thm:dimensiobee}
  Suppose $t \ge 0$ and $\mu$ is an ergodic $t$--equilibrium
  measure. Then
  \begin{equation}
    \lim_{n \to \infty}
    \frac{\log\mu([\iii|_n])}{\log\psi^t_{\iii|_n}(\hhh)} = 1 -
    \frac{P(t)}{E_\mu(t)}
  \end{equation}
  for $\mu$--almost all $\iii \in I^\infty$.
\end{proposition}

\begin{proof}
  Let us first note that due to the invariance of the equilibrium
  measure and theorem of Shannon--McMillan (for example, see
  Chapter 3 of Zinsmeister \cite{zi}) we have
  \begin{equation} \label{eq:entroesitys}
    h_\mu = -\lim_{n \to \infty} \tfrac{1}{n} \log\mu([\iii|_n])
  \end{equation}
  for $\mu$--almost all $\iii \in I^\infty$. We can get a similar kind
  of expression for the energy as well. Indeed, using
  Kingman's subadditive ergodic theorem (for example, see Steele
  \cite{st}) and the BVP, we have
  \begin{align}
    E_\mu(t) &= \lim_{n \to \infty} \tfrac{1}{n} \sum_{\iii \in I^n}
    \mu([\iii])\log\psi^t_\iii(\hhh) \notag \\
    &= \lim_{n \to \infty} \tfrac{1}{n} \sum_{\iii \in I^n}
    \int_{[\iii]} \log\psi^t_\iii(\hhh) d\mu(\hhh) \notag \\
    &= \lim_{n \to \infty} \tfrac{1}{n} \int_{I^\infty}
    \log\psi^t_{\iii|_n} \bigl( \sigma^n(\iii) \bigr) d\mu(\iii) \\
    &= \lim_{n \to \infty} \tfrac{1}{n} \log\psi^t_{\jjj|_n}
    \bigl( \sigma^n(\jjj) \bigr) \notag \\
    &= \lim_{n \to \infty} \tfrac{1}{n} \log\psi^t_{\jjj|_n}(\hhh) \notag
  \end{align}
  for $\mu$--almost all $\jjj \in I^\infty$. Now the claim follows
  easily from the fact
  \begin{equation}
    P(t) = E_\mu(t) + h_\mu.
  \end{equation}
\end{proof}

Now, with the help of this proposition, we can prove the next theorem,
our main tool in studying the Hausdorff dimension of the equilibrium
measure on affine systems.
We define the \emph{equilibrium dimension of a measure $\mu \in
\MM(I^\infty)$} by setting $\dime(\mu) = \inf\{ \dime(A) : A \text{ is
a Borel set such that } \mu(A)=1 \}$.

\begin{theorem} \label{thm:dimensiocee}
  Suppose $P(t)=0$ and $\mu$ is an ergodic $t$--equilibrium
  measure. Then
  \begin{equation}
    \dime(\mu) = t.
  \end{equation}
\end{theorem}

\begin{proof}
  Let us denote
  \begin{equation}
    R = \Bigl\{ \iii \in I^\infty : \lim_{n \to \infty}
    \frac{\log\mu([\iii|_n])}{\log\psi^t_{\iii|_n}(\hhh)} = 1 \Bigr\}
  \end{equation}
  and take an arbitrary Borel set $A \subset I^\infty$ for which
  $\mu(A)=1$. Using Proposition \ref{thm:dimensiobee}, we also have
  $\mu(R \cap A) = 1$. Fix $\iii \in R \cap A$ and $q<t$. Now it
  follows from the definition of the cylinder function, Proposition
  \ref{thm:dimensiobee} and (\ref{eq:entroesitys}) that
  \begin{align}
    \liminf_{n \to \infty}\,
    \frac{\log\mu([\iii|_n])}{\log\psi_{\iii|_n}^q(\hhh)} &\ge
    \lim_{n \to \infty}
    \frac{\log\mu([\iii|_n])}{\log\psi_{\iii|_n}^t(\hhh) + \log
    \ys_{t-q}^{-n}} \notag \\
    &= \frac{1}{1 + \tfrac{1}{h_\mu}\log \ys_{t-q}} > 1.
  \end{align}
  Thus there exists $n_0 = n_0(\iii)$ such that
  \begin{equation} \label{eq:kohtayksi}
    \frac{\log\mu([\iii|_n])}{\log\psi_{\iii|_n}^q(\hhh)} \ge 1
  \end{equation}
  whenever $n \ge n_0$. Denoting $A_k = \{ \iii \in R \cap A :
  n_0(\iii) < k \}$, we have $R \cap A = \bigcup_{k=1}^\infty
  A_k$. Hence, using (\ref{eq:kohtayksi}), we get for each $\iii \in
  A_k$
  \begin{equation} \label{eq:kohtakaksi}
    \frac{\mu([\iii|_n])}{\psi_{\iii|_n}^q(\hhh)} \le 1
  \end{equation}
  whenever $n \ge k$. Take $\{ [\iii_j] \}_j$ to be any cover for
  $A_k$ such that $|\iii_j| > k$ and $[\iii_j] \cap A_k \ne \emptyset$
  for every $j$. We can choose each $\iii_j$ to be of the form
  $\iii|_n$ for some $\iii \in A_k$ and $n \in \N$. Hence by
  (\ref{eq:kohtakaksi})
  \begin{equation}
    \mu(A_k) \le \sum_j \mu([\iii_j]) \le \sum_j \psi_{\iii_j}^q(\hhh)
    \le K_q \sum_j \int_{I^\infty} \psi_{\iii_j}^q(\hhh) d\mu(\hhh),
  \end{equation}
  from which we get $\GG^q(A_k) \ge K_q^{-1}\mu(A_k)$. Now, clearly,
  $\GG^q(R \cap A) = \lim_{k \to \infty}\GG^q(A_k) \ge K_q^{-1}
  \lim_{k \to \infty} \mu(A_k) = K_q^{-1}\mu(R \cap A)$, which gives
  $\GG^q(A) > 0$ and $\dime(A) \ge q$. Since $q<t$ was arbitrary
  as was the choice of the Borel set $A$ of full measure, we conclude
  $\dime(\mu) \ge t$. The proof is finished by recalling Theorem
  \ref{thm:eqdim}.
\end{proof}

In the similitude and conformal cases we obtained the desired dimension result
easily straight from Theorem \ref{thm:stabledim}. For the affine IFS we
can not apply Theorem \ref{thm:stabledim} because in that case it gives
only upper and lower bounds for the Hausdorff dimension of the
equilibrium measure. We will use Theorem \ref{thm:dimensiocee} and the
following result of Falconer \cite{fa1}.

\begin{theorem} \label{thm:falconer}
  Suppose mappings of an affine IFS are of the form $\fii_i(x) = A_ix + a_i$,
  where $|A_i| < \tfrac{1}{3}$, as $i \in I$ and the cylinder function is
  chosen to be the singular value function, $\psi^t_\iii \equiv
  \alpha^t(A_\iii)$. We also assume that $P(t)=0$. Then for
  $\HH^{d\#I}$--almost all $a=(a_1,\ldots,a_{\#I}) \in \R^{d\#I}$ we
  have
  \begin{equation}
    \dime(I^\infty) = \dimh(E)
  \end{equation}
  where $E = E(a)$.
\end{theorem}

The main idea of the proof is to use ellipsoids as a covering. Since
the singular value function refers to the size of the corresponding
ellipsoid, this is natural. The upper bound for the Hausdorff
dimension is a straightforward calculation and the lower bound is
obtained using the potential theoretic characterisation of the
Hausdorff dimension. Solomyak has improved the constant $\tfrac{1}{3}$
used in the theorem. He proved that it can be replaced by
$\tfrac{1}{2}$, which, rather surprisingly, he showed to be sharp
in a sense if $|A_i| \ge \tfrac{1}{2} + \eps$ for some $i \in I$ and
for any $\eps > 0$, then the theorem may fail. For details see
Proposition 3.1 of 
\cite{so}. Falconer's theorem is true also for subsets of $E$,
that is, for $\HH^{d\#I}$--almost all $a$ we have $\dime\bigl(
\pi^{-1}(A) \bigr) = \dimh(A)$ whenever $A 
\subset E = E(a)$ is a Borel set. This generalisation follows just by
noting that Lemma 4.2 of \cite{fa1} remains true if the set $I^\infty$
is replaced by an arbitrary Borel set.

Notice that in the
theorem no separation condition of any kind is assumed. However, there are
situations where the equilibrium dimension and the Hausdorff dimension do
not coincide if we just assume $|A_i| < \tfrac{1}{2}$ for every $i \in
I$. For example, there is too much overlapping among the sets
$\fii_\iii(X)$, or these sets are aligned in a way that it is not
possible to obtain economical covers using ellipsoids, and, thus, the
use of the singular value function does not fit. In the theorem all of
these ``bad'' situations are excluded by the statement ``for
$\HH^{d\#I}$--almost all $a$''. It is an interesting question to find
a characterisation for these ``bad'' situations. Hueter and Lalley have
provided in \cite{hl} with checkable sufficient conditions for the theorem
to hold for all $a$.

The following theorem gives a
partially positive answer to the open question proposed by Kenyon and
Peres in \cite{kp}. They asked whether there exists a $T$--invariant
ergodic probability measure on a given compact set, where the mapping
$T$ is continuous and expanding, such that it has full dimension. In
our case the mapping $T$ is constructed by using inverses of the
mappings of IFS.

\begin{theorem}
  Suppose mappings of an affine IFS are of the form $\fii_i(x) = A_ix + a_i$,
  where $|A_i| < \tfrac{1}{2}$, as $i \in I$ and the cylinder function is
  chosen to be the singular value function, $\psi^t_\iii \equiv
  \alpha^t(A_\iii)$. We also assume that $P(t)=0$, $\mu$ is an ergodic
  $t$--equilibrium measure and $m = \mu \pallo \pi^{-1}$. Then for
  $\HH^{d\#I}$--almost all $a=(a_1,\ldots,a_{\#I}) \in \R^{d\#I}$ we
  have
  \begin{equation}
    \dimh(m) = \dimh(E),
  \end{equation}
  where $E = E(a)$.
\end{theorem}

\begin{proof}
  Due to Theorems \ref{thm:dimensiocee} and \ref{thm:eqdim} we have
  $\dime(A) = \dime(I^\infty)$ whenever $A \subset I^\infty$ has full
  $\mu$--measure. Hence for any $A \subset E$ with full $m$--measure
  we have
  \begin{equation}
    \dimh(A) = \dime\bigl( \pi^{-1}(A) \bigr) = \dime(I^\infty) = \dimh(E)
  \end{equation}
  using Theorem \ref{thm:falconer} and the comments after it.
\end{proof}



\end{document}